\documentclass{article}
\usepackage[utf8]{inputenc}
\usepackage{authblk}
\usepackage{setspace}
\usepackage[margin=1in]{geometry}
\usepackage{graphicx}
\graphicspath{ {./figures/} }
\usepackage{subcaption}
\usepackage{lineno}
\usepackage{comment}
\usepackage{amsthm}
\usepackage{amsmath}
\newtheorem{theorem}{Theorem}
\usepackage{amsfonts} 
\usepackage{mathtools} 
\usepackage{caption}
\usepackage{multirow}

\usepackage{biblatex}
\usepackage{apalike}
\bibliography{ref.bib}
\usepackage{url}
\usepackage{hyperref}
\usepackage{apalike}

\usepackage{xcolor}
\usepackage{hyperref}
\hypersetup{
    colorlinks,
    linkcolor={blue!95!black},
    citecolor={blue!95!black},
    urlcolor={blue!95!black}
}

\title{\huge Optimal Intervention Strategies and Cost-effectiveness Analysis study of Tuberculosis with reference to  TPT, Malnutrition and Diabetes Management }

\author[a]{\textit{Mr. Sushil Chhetri}}
\author[*a,b]{\textit{Dr. Krishna Kiran Vamsi Dasu}}
\author[ c]{\textit{Mrs. K N Kavya}} 
\author[ d,b]{\textit{Dr. Sharath B N}} 
\author[e,b]{\textit{Dr. Uma Shankar S}} 
\author[e,b]{\textit{Dr. Somashekar N}}
\author[ e,b ]{\textit{Dr. Vineet Kumar Chadda}}

\affil[a]{ \textit{Department of Mathematics and Computer Science, Sri Sathya Sai Institute of Higher Learning, Puttaparthi, India.}}
\affil[b]{\textit{Center for Excellence in Mathematical Biology, SSSIHL,  Puttaparthi, India}}
\affil[c]{\textit{Department of Mathematics, CHRIST (Deemed to be University), Bengaluru, India}}
\affil[d]{  \textit{ESI Medical College, Bengaluru, Karnataka, India}}
\affil[e]{ \textit{National Tuberculosis Institute - NTI, Bengaluru, Karnataka, India}}
\affil[*]{  \textit{Corresponding author}}
\affil[  ]{ \textit{sushilchhetri@sssihl.edu.in, dkkvamsi@sssihl.edu.in, sharath.burug@gmail.com, shankars.2017@gmail.com, bsoma@gmail.com, vineet2chadha@gmail.com}}

\date{}

\begin{document}

\maketitle

\begin{abstract} {{
Tuberculosis remains a significant global health challenge, with millions of new cases reported annually. Recent studies suggest that expanding the accessibility of TB intervention programs can lead to a substantial decrease in both TB incidence and prevalence. This paper initiates by examining a deterministic mathematical model for TB transmission, aiming to analyze the underlying dynamics. Subsequently, an optimal control problem is formulated to enhance TB control measures, encompassing Tuberculosis Preventive Treatment (TPT) and other initiatives targeting malnutrition and diabetes. Through simulation studies, the effectiveness of the control program is assessed. The model dynamics allow us to identify the pseudo-prevalence and incidence. To determine the potential long-term trajectory of TB and to acquire future projections a cost-effectiveness analysis is performed using ACER, AIR, ICER, and four quadrants to compare competing interventions. In conclusion, this work   provides valuable insights into TB and strategies for its control and cost effectiveness.}}
\end{abstract}
{ \bf {keywords:} } Optimal control; Incidence; Pseudo-prevalence; Hamiltonian;Cost effectiveness;Threshold


\section{Introduction} \label{Intro}

 Tuberculosis is caused by tubercle bacilli, a vital and enduring global health disease transferred through the air. The bacterium is mostly pathogenic in the lung, and its symptom is coughing. When individuals become infected, they enter a latent stage, the duration of which varies from person to person. For many, this latent period can span 20 to 30 years, during which they harbor the M. tuberculosis bacterium without developing active TB. Latently infected individuals with TB are asymptomatic and do not transmit the disease, but they may transition to active TB through either endogenous reactivation or exogenous reinfection \cite{small2001management}.  \\


Even today the global burden of TB remains substantial. According to the World Health Organization's Global Tuberculosis Report of 2016, TB ranks among the top 10 causes of mortality worldwide, claiming over 1.7 million lives annually \cite{TBREPORT}. The worldwide incidence of TB in 2019 numbered approximately 10 million \cite{world2020end}. Shockingly, this ancient disease continues to afflict millions globally, with an estimated one-third of the world's population carrying latent Mycobacterium tuberculosis infections, as highlighted in \cite{bleed2001world}. The National TB prevalence survey conducted in India from 2019 to 2021 estimated a crude prevalence of TB infection ($T_l$) at 31.3\% among individuals aged 15 years and above \cite{TBREPORT}. According to WHO Global TB report 2022, the incidence rate in India was 210 (178-244) per 100,000 population, and according to the in-country model, the incidence rate was 196 (171-228) per 100,000 population. The TB notification rate was 172 per 100,000 population, and TB treatment success rate was 85\%, respectively. \\

The co-infection of tuberculosis and diabetes is a significant concern, with historical recognition dating back to the late 1883 . The study  by Oscarsson \cite{oscarsson1958ii}, underscores the prevalence of this co-infection, indicating that around 50\% of individuals with tuberculosis also had diabetes. This finding highlights the importance of considering diabetes as a comorbidity in the management and treatment of tuberculosis, as it can influence the course and outcome of the disease. In the study \cite{jacob2023cost}, it was found that the median yearly direct expenditure for managing diabetes among the participants amounted to INR 7540 in india. Similarly, the relationship between malnutrition and tuberculosis is closely intertwined. Malnutrition weakens the immune system, making individuals more susceptible to tuberculosis infection and hindering their ability to fight off the disease. Conversely, tuberculosis can exacerbate malnutrition by causing weight loss, decreased appetite, and nutrient malabsorption. This vicious cycle of malnutrition and tuberculosis creates a challenging scenario where addressing both conditions concurrently is essential for successful treatment outcomes. \\

In  this paper, our initial focus is on a deterministic mathematical model for tuberculosis (TB) transmission. We conduct preliminary analysis to gain insights into the fundamental dynamics of the disease. The optimal control problem is subsequently formulated and examined, which includes the implementation of TB control intervention strategies like Tuberculosis Preventive Treatment (TPT), as well as other TB control programs addressing malnutrition and diabetes. Our objective is to gain insights into the potential of these programs in mitigating the transmission of tuberculosis and enhancing overall public health. The pseudo-prevalence and incidence are determined by analyzing model dynamics. By extrapolating these values to future scenarios, valuable insights can be obtained regarding the long-term trajectory of tuberculosis (TB).
 \\

In addition, a cost-effectiveness analysis is conducted to obtain a thorough comprehension of the economic consequences associated with the implementation of interventions for tuberculosis (TB) control. This analysis facilitates the assessment of costs linked to different strategies for tuberculosis (TB) prevention, diagnosis, and treatment, in comparison to their efficacy in mitigating the burden of TB. The objective of conducting cost-effectiveness analysis is to provide valuable insights for decision-making processes. This analysis will help guide the allocation of limited resources in order to optimize the impact on tuberculosis control efforts and ultimately enhance public health outcomes.
\\

The organization of this paper is as follows. In the next section we frame the initial model and discuss its positivity \& boundedness and later do the stabilty analyis for the same. Further in section 3 we frame an optimal control problem based on the initial model and the  control interventions. We perform the numerical studies and propose the optimal control strategies. In section 4 we calculate the Pseudo-Prevalence and Incidence. Further in section 5 we perform the cost-effective analysis and end the work with discussions and conclusions section.

\section{The Initial Model }
Mathematical modelling plays a crucial role in understanding the complexities of infectious diseases. They provide researchers 
  with the ability to model a variety of scenarios and make predictions regarding the spread of illnesses, which significantly aids in the development of efficient control techniques. These models help to make educated decisions regarding public health by providing insights into the progression of diseases and incorporating data on population dynamics, transmission rates, and other vital aspects. \\

\par The present study partitions the population into five distinct compartments, which helps to understand the population dynamics comprehensively. The symbol $U(t)$ represents the total population susceptible to tuberculosis. This population comprises a complex blend of individuals at risk of infection, as well as those with the potential to develop a sick or latent state. $TBI(t)$ are the population with asymptomatic TB infection, posing no transmission risk. Conversely, individuals actively exhibiting TB symptoms and capable of transmitting the disease constitute the $TBD(t)$ compartment. Those undergoing TB treatment are represented by $TT(t)$, which includes active cases receiving medical care. Finally, individuals who have completed treatment and are TB-free are denoted by the $R(t)$ compartment.\\

The total population at a time t is given by $N(t) = U(t) + TBI(t) + TBD(t) + TT(t) + R(t)$.The transmission dynamics of the tuberculosis is presented below. In this flow diagram we also include the control interventions such as TPT ($\mu_{TPT}(t)$), Malnutrition ($\mu_{MN}(t)$) and Diabetes  ($\mu_{D}(t)$) management.  To make the notations simpler the variables are renamed as follows:
 $TBI = T_l$, $TBD = T_d$, $TT = T_t$, $\mu_{TPT} = \mu_T$ and $\mu_{MN} = \mu_M$ respectively.

\begin{figure}[ht] 
    \centering
    \includegraphics[width=12cm, height=6cm]{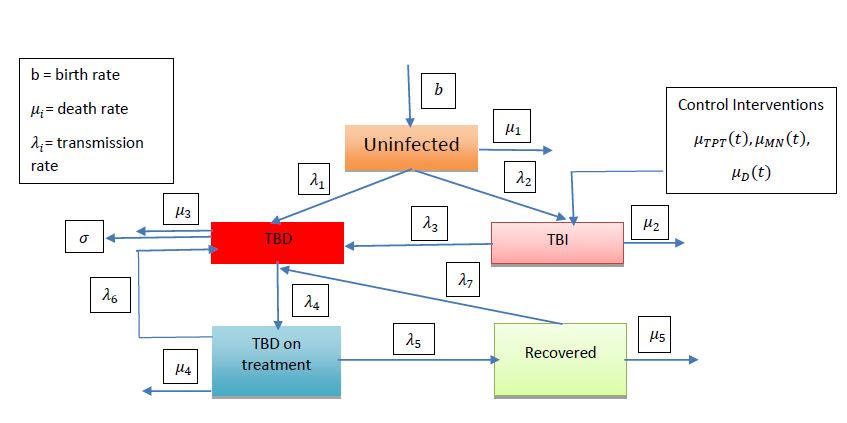}
    \caption{Disease flow diagram for the transmission dynamics of tuberculosis.}
\end{figure}

  Table \ref{tab:variables-parameters} outlines key parameters, control variables and initial conditions crucial for modeling the dynamics of tuberculosis (TB). These parameters include rates such as $\lambda_1$ and $\lambda_2$, denoting the likelihood of uninfected individuals acquiring TB infection and developing active disease, respectively. Additionally, rates like $\lambda_3$ illustrate the progression from latent TB infection to active disease, while $\lambda_4$ and $\lambda_5$ signify the rates of initiating treatment for TB and recovering from the disease during treatment. Other factors, such as the default rate from TB treatment $\lambda_6$ and the recurrence rate of TB in recovered individuals $\lambda_7$, are also included. Complementing these rates are variables like the natural birth rate (b), the natural cure rate of TB ($\sigma$), and death rates ($\mu_1-\mu_5$) for various disease compartments. Moreover, initial population counts for different stages of TB, including uninfected, infected, diseased, under treatment, and recovered individuals, provide a foundational framework for simulating TB transmission and treatment dynamics.\\
  
  Following the above descriptions,  the differential equations representing the five compartmental deterministic models used in studying the dynamics of TB is represented as
\numberwithin{equation}{section} 
\begin{align}
   & \frac{dU}{dt} = b - (\lambda_1 + \mu_1)U - \lambda_2 T_l U && \label{begineq}\\
   & \frac{dT_l}{dt} = \lambda_2T_l U - (\lambda_3 + \mu_{2})T_l && \ \label{sec11equ21} \\
   & \frac{dT_d}{dt} = \lambda_1U + \lambda_3T_l - (\lambda_4 + \mu_3 + \sigma)T_d + \lambda_6T_t + \lambda_7R &&  \label{sec11equ31}\\
   & \frac{dT_t}{dt} = \lambda_4T_d - (\lambda_5 + \lambda_6 + \mu_4)T_t &&  \label{sec11equ45} \\
   & \frac{dR}{dt} = \lambda_5 T_t - (\lambda_7 + \mu_5)R &&  \label{endeq}
\end{align}
with the initial conditions U(0)$>$0, $T_l$ (0) $>$ 0, $T_d(0)$ $>$ 0, $T_t(0)$ $>$ 0 and R(0) $>$ 0.

\begin{table}[ht]
   \renewcommand{\arraystretch}{0.5}  
   \begin{tabular}{|c|p{130mm}|}
    \hline
     \textbf{Parameter} & \textbf{Biological Meaning}  \\
     \hline
      U &  Uninfected Population\\
         \hline
      $T_l$ &Tuberculosis latently infected population\\
      \hline
      $T_d$ & Tuberculosis diseased population\\
      \hline
      $T_t$ & Tuberculosis diseased population undergoing treatment\\
      \hline
      $R$ & Recovered tuberculosis diseased population\\
      \hline
       MN&Malnutrition\\
       \hline
      TPT&Tuberculosis preventive treatment.\\
      \hline
      D &Diabetes\\
      \hline
      $\lambda_1$,$\lambda_2$ & Rate at which an uninfected individual getting TBI,TBD respectively.\\
      \hline
      $\lambda_3$ & Rate at which a latently TB infected individual getting TBD.\\
     \hline
      $\lambda_4$& Rate at which a TBD individual is put on treatment.\\
      \hline
      $\lambda_5$ & Rate at which a TBD individual on treatment 
      gets recovered.\\
      \hline
      $\lambda_6$& Rate of default of an individual from TBD treatment\\
      \hline
      $\lambda_7$& Rate at which a recovered individual getting TBD (Recurrence rate)\\
      \hline
      $b$ & Natural birth rate\\
      \hline
      $\sigma$ &Natural cure rate of an individual having TBD.\\
      \hline
      $\mu_{TPT}(t)$ & Rate at which infected population is decreased  continuously due to TPT given to TB  infected.\\
      \hline
      $\mu_{MN}(t)$&  Rate at which infected population is decreased  continuously due to  mal-nutrition interventions\\
        \hline
      $\mu_D(t)$ & Rate at which infected population is decreased continuously due diabetic interventions.\\
        \hline
      $\mu_1,\mu_2,\mu_3,\mu_4,\mu_5$& Death rate of an individual in compartments $U, T_l,T_d,T_t,R $\newline
      respectively.\\
        \hline
      U(0) & Uninfected population at the initial time.\\
        \hline
      $T_l$(0)   & TB latently infected population at the initial time.\\
    
        \hline
      $T_d$(0)& TB diseased population at the initial time.\\
        \hline
      $T_t$(0) &Population undergoing treatment at the initial time.\\
        \hline
      R(0)  & TB disease recover population at the initial time.\\
       \hline
      \end{tabular}
      \caption{Description of variables and parameters}
    \label{tab:variables-parameters}
\end{table}

\subsection{Positivity and Boundedness of Solutions }
 To ensure the meaningfulness of the model described by equations (\ref{begineq})-(\ref{endeq}), it is imperative to demonstrate the system's positivity and the boundedness of its solutions. This requires proving that when the initial conditions are positive, the solutions remain positive for all future time intervals. The methodology for establishing the positivity and boundedness closely resembles that discussed in \cite{chhetri2021within}. \\

{\large{\bf{Positivity:}}} \\

\begin{align*}
\begin{split}
\frac{dU}{dt}\bigg|_{U = 0} &= b > 0, \quad
\frac{dT_l}{dt}\bigg|_{T_l = 0} = 0 \geq 0, \\
\frac{dT_d}{dt}\bigg|_{T_d = 0} &= \lambda_1U + \lambda_3T_l + \lambda_6T_t + \lambda_7R \geq 0,
\end{split} \\
\frac{dT_t}{dt}\bigg|_{T_t = 0} &= \lambda_4T_d \geq 0,\quad
\frac{dR}{dt}\bigg|_{R = 0} = \lambda_5T_t \geq 0
\end{align*}
 Based on the analysis of the rates given by equations (\ref{begineq})-(\ref{endeq}) on the bounding planes within the non-negative region of the real space, it is evident that all these rates are non-negative. This observation implies that the direction of the vector field is consistently inward on these bounding planes, as indicated by the respective inequalities. Consequently, if a solution initiates from the interior of this region, it will remain confined within it for all time intervals t. This is because the vector field always points inward on the bounding planes. Therefore, we deduce that all solutions of the system represented by equations (\ref{begineq})$-$(\ref{endeq}) maintain positivity for any t $>$0, provided that the initial conditions are positive. \\
 
{\large{\bf{Boundedness:}}} \\

 \indent To establish the boundedness of the solutions to (\ref{begineq})-(\ref{endeq}) we have, the total population as\\ \indent \textit{ N = U + $T_l$ + $T_d$+ $T_t$  + R }. Now, 
 \begin{align*}
 \frac{dN}{dt} &= \frac{dU}{dt} + \frac{dT_l}{dt} + \frac{dT_d}{dt} + \frac{T_t}{dt} + \frac{dR}{dt} \\
 &\leq b-\omega (\textit{U + $T_l$ + $T_d$ + $T_t$ + R }) \\
 &\leq b-\omega N
 \end{align*}
   where,  $\omega = \min\{ \kappa, \mu_1, \mu_2, \mu_4, \mu_5 \} \hspace {1em }and \hspace{1em} \kappa =\min\{\mu_3 + \sigma\}.$ 
 After we integrate, we have 
 \begin{align*}
 N(t)&= \frac{b}{\omega} +Ce^{-\omega t}
 \end{align*}
   where, C is constant. Now  as ${t \to \infty}$,
 \begin{align*}
   lim \hspace{0.5em } sup\hspace{0.5em } N(t)&\leq  \frac{b}{\omega}
 \end{align*}

It is clear that the system (\ref{begineq})-(\ref{endeq}) is positive and bounded.Therefore, all
the solutions of system  with non negative initial values in
the space ${R_+}^5$ are bounded and exist on the interval $[0, \infty)$.\\

Therefore, the biologically feasible region of the model is the following positive invariant set:
\begin{align*}
 \Omega = \left\{ (U(t), T_l(t), T_d(t), T_t(t), R(t)) \ |  N(t) \leq \frac{b}{\omega}, t \geq 0 \right\}
\end{align*}
Based on the above results on positivity and boundedness of the system (\ref{begineq})-(\ref{endeq}), we have the following theorem.

\begin{theorem} \label{th1}
The set $ \Omega = \left\{ (U(t), T_l(t), T_d(t), T_t(t), R(t)) \ \middle|  N(t) \leq \frac{b}{\omega}, t \geq 0 \right\}$ is a positive invariant and an attracting set for system (\ref{begineq})--(\ref{endeq}).
\end{theorem}

\subsection{Disease-Free Equilibrium}
The TB-free equilibrium represents a critical state in epidemiological models where the population is entirely devoid of tuberculosis infection. This equilibrium is attained by setting the right-hand side of the model (\ref{begineq})-(\ref{endeq}) to zero and nullifying the variables, $T_l = 0$, $T_d = 0$ and $T_t = 0$. For the model (\ref{begineq})$-$(\ref{endeq}), the disease free equillibrium is given by 
\[
\textbf{E}_0 = \Bigg(\frac{b}{\lambda_1 + \mu_1},0,0,0,0\Bigg)
\]

\subsection{Infected Equilibrium}
The infected equilibrium in a tuberculosis mathematical model is a crucial state where the transmission of the disease reaches a steady balance between new infections and recoveries. At this equilibrium, the number of individuals actively infected with tuberculosis remains constant over time. Mathematically, it is identified by setting the rate of change of all the compartments to zero. Setting the right hand side of (\ref{begineq})-(\ref{endeq}) to zero we get,
\[
\begin{aligned}
&\textbf{$U$}^* = \frac{A}{\lambda_2}, \\
&\textbf{$T_l$}^* = \frac{b\lambda_2 - (\lambda_1 + \mu_1)A}{A \lambda_2}, \\
&\textbf{$T_d$}^* = \frac{B\, \textbf{$T_t$}^*}{\lambda_4}, \\
&\textbf{$T_t$}^* = \frac{x}{y}, \\
&\textbf{$R$}^* = \frac{\lambda_5 \, \textbf{$T_t$}^*}{\lambda_7 + \mu_5},\\
where,  
    &\indent A = \lambda_3 + \mu_2, \quad B = \lambda_5 + \lambda_6 + \mu_4, \quad C = \lambda_4 + \mu_3 + \sigma, \quad D = b\lambda_2 - (\lambda_1 + \mu_1)A, \\
    &\indent x = -\left(\frac{\lambda_1 A}{\lambda_2} + \frac{\lambda_3 D}{\lambda_2 A}\right) \quad \text{and} \quad y = \left(\frac{\lambda_7 \lambda_5}{\lambda_7 + \mu_5} - \frac{B C}{\lambda_4} + \lambda_6\right).
\end{aligned}
\]

\subsection{Calculation of \texorpdfstring{$ R_0 $}{ }}

A basic reproduction number $R_0$ represents the average number of secondary infections produced by a single infected individual in a completely susceptible population. If $R_0$ is greater than 1, it suggests that the disease is likely to spread within the population, leading to an epidemic or outbreak. Conversely, if $R_0$ is less than 1, the disease is likely to die out over time, as each infected individual, on average, infects fewer than one other person.
  
  We calculate the reproduction number using the next-generation matrix method, as outlined in \cite{diekmann2010construction}, and obtain Jacobian matrices for new infection terms \( F \) and remaining transfer terms as \( V \).
 \begin{align*}
    F &= \begin{bmatrix}
        \lambda_2 T_l U \\
        \lambda_1 U \\
        0 
    \end{bmatrix}  &
    V &= \begin{bmatrix}
        -(\lambda_3 + \mu_{2})T_l \\
        \lambda_3T_l - (\lambda_4 + \mu_3 + \sigma)T_d + \lambda_6T_t + \lambda_7R \\
        \lambda_4T_d - (\lambda_5 + \lambda_6 + \mu_4)T_t  
    \end{bmatrix}
 \end{align*}
 The jacobian matrix of F and V matrices with respect to state variables $T_l$, $T_d$ and $T_t$ needs to be found out.
 Let the Jacobian matrix of F and V  be $F_j$  and  $V_j$ respectively. 

 \begin{align}
 F_j &=  \begin{bmatrix}
 \lambda_2 U & 0 & 0 \\
 0 & 0 & 0 \\
 0 & 0 & 0
 \end{bmatrix}  \quad\quad\quad\quad\quad &
 V_j& =  \begin{bmatrix}
 -(\lambda_3 + \mu_2) & 0 & 0 \\
 \lambda_3 & -(\lambda_4 + \mu_3 + \sigma) & \lambda_6 \\
 0 & \lambda_4 & -(\lambda_5 + \lambda_6 + \mu_4)
 \end{bmatrix}
 \end{align}

 The largest eigenvalue of the matrix,  $F_j V_j^{-1}$ at $\textbf{E}_0 $   is basic reproduction number and is given by 
\[
\textbf{R}_0 = \frac{\lambda_2 b}{(\lambda_1 + \mu_1)(\lambda_3 +\mu_2)}
\]
 \subsection{Existence and Uniqueness of  Solution}
 It is very important to establish the existence of solution to the system (\ref{begineq})-(\ref{endeq}) before any further analysis of the model.
 In this section, following the approach as in  \cite{sowole2019existence}, we investigate the existence and uniqueness of solutions for our system (\ref{begineq})-(\ref{endeq}). We delve into the analysis of a general first-order ordinary differential equation  of the form
 \begin{equation}
 \dot{x} = f(t, x), \quad x(t_0) = x_0
 \end{equation}
 with $f : \mathbb{R} \times \mathbb{R}^n \rightarrow \mathbb{R}^n$
 sufficiently many times differentiable.\\
 
 The key questions revolve around the conditions under which a solution exists and when it is unique. We employ a theorem discussed in \cite{sowole2019existence} to establish the existence and uniqueness of solutions for our model.

\begin{theorem} \label{th2}
    Let $D$ be the domain defined as follows:
    \[
    |t - t_0| \leq a, \quad \|x - x_0\| \leq b, \quad x = (x_1, x_2, \dots, x_n), \quad x_0 = (x_{10}, x_{20}, \dots, x_{n0})
    \]
    Assume $f(t, x)$ satisfies the Lipschitz condition:
    \[
    \| f(t, x_2) - f(t, x_1) \| \leq k \| x_2 - x_1 \|
    \]
    for any pairs $(t, x_1)$ and $(t, x_2)$ in the domain $D$, where $k$ is a positive constant. Then, there exists a constant $\delta > 0$ such that a unique continuous vector solution $x(t)$ exists.
\end{theorem}
It's important to note that condition (7) is met by ensuring that the partial derivatives $\frac{\partial f_i}{\partial x_j}$, for $i,j=1,2,3,\dots,n$, are continuous and bounded within the domain $D$.

\begin{theorem} \label{th3}
(Existence of Solution) Let $D$ be the domain defined as above, satisfying the conditions. Then, there exists a unique solution for the system \eqref{begineq}-\eqref{endeq} within the domain $D$, which remains bounded..
\end{theorem}
\begin{proof} Let
\numberwithin{equation}{section}
\begin{align}
 & f_1 = b - (\lambda_1 + \mu_1)U - \lambda_2 T_l U && \label{sec11equ1}\\
  &f_2 = \lambda_2T_l U - (\lambda_3 + \mu_{2})T_l && \ \label{sec11equ2} \\
&f_3 = \lambda_1U + \lambda_3T_l - (\lambda_4 + \mu_3 + \sigma)T_d + \lambda_6T_t + \lambda_7R &&  \label{sec11equ32}\\
 & f_4 = \lambda_4T_d - (\lambda_5 + \lambda_6 + \mu_4)T_t &&  \label{sec11equ4} \\
& f_5 = \lambda_5 T_t - (\lambda_7 + \mu_5)R &&  \label{sec11equ5}
\end{align}
We will show that,\newline
 \indent \hspace{14em} $\displaystyle \frac{\partial f_i}{\partial x_j}$, i, j=1,2,3,....,n,\newline is continuous and bounded in the domain D.\newline 
 Considering  (8), we have, \hspace{1em} $\displaystyle \frac{\partial f_1}{\partial U}$  $=$ $| -(\lambda_1+\mu_1)-\lambda_2T_l | < \infty$,\quad
$\displaystyle \frac{\partial f_1}{\partial T_l}$  $=$ $| -\lambda_2 U | < \infty$,\newline\\
\indent \hspace{11em} $\displaystyle \frac{\partial f_1}{\partial T_d}$  $=$ $0 < \infty$,\quad
$\displaystyle \frac{\partial f_1}{\partial T_t}$  $=$ 0 $<$ $ \infty$,\quad
$\displaystyle \frac{\partial f_1}{\partial R}$  $=$ 0 $<$ $\infty$,\newline

Considering  (9), \hspace{3.5em} $\displaystyle \frac{\partial f_2}{\partial U}$  $=$ $\lambda_2T_l < \infty$,\quad
$\displaystyle \frac{\partial f_2}{\partial T_l}$  $=$ $| \lambda_2 U-(\lambda_3+\mu_2) | < \infty$,\newline\\
\indent \hspace{11em} $\displaystyle \frac{\partial f_2}{\partial T_d}$  $=$ $0 < \infty$,\quad
$\displaystyle \frac{\partial f_2}{\partial T_t}$  $=$ 0 $<$ $ \infty$,\quad
$\displaystyle \frac{\partial f_2}{\partial R}$  $=$ 0 $<$ $\infty$,\newline 

Considering  (10), \hspace{3em} $\displaystyle \frac{\partial f_3}{\partial U}$  $=$ $\lambda_1  < \infty$,\quad
$\displaystyle \frac{\partial f_3}{\partial T_l}$  $=$ $\lambda_3  < \infty$,\newline\\
\indent \hspace{11em} $\displaystyle \frac{\partial f_3}{\partial T_d}$  $=$ $|-(\lambda_4+\mu_3+\sigma)| < \infty$,\quad
$\displaystyle \frac{\partial f_3}{\partial T_t}$  $=$ $\lambda_6$ $<$ $ \infty$,\quad
$\displaystyle \frac{\partial f_3}{\partial R}$  $=$ $\lambda_7$ $<$ $\infty$,\newline 

Considering  (11), \hspace{3em} $\displaystyle \frac{\partial f_4}{\partial U}$  $=$ $0  < \infty$,\quad
$\displaystyle \frac{\partial f_4}{\partial T_l}$  $=$ $0 < \infty$,\newline\\
\indent \hspace{11em} $\displaystyle \frac{\partial f_4}{\partial T_d}$  $=$ $\lambda_4 < \infty$,\quad
$\displaystyle \frac{\partial f_4}{\partial T_t}$  $=$ $|-(\lambda_5+\lambda_6+\mu_4)|$ $<$ $ \infty$,\quad
$\displaystyle \frac{\partial f_4}{\partial R}$  $=$ $0$ $<$ $\infty$,\newline 

Considering  (12), \hspace{3em} $\displaystyle \frac{\partial f_5}{\partial U}$  $=$ $0  < \infty$,\quad
$\displaystyle \frac{\partial f_5}{\partial T_l}$  $=$ $0 < \infty$,\newline\\
\indent \hspace{11em} $\displaystyle \frac{\partial f_5}{\partial T_d}$  $=$ $0 < \infty$,\quad
$\displaystyle \frac{\partial f_5}{\partial T_t}$  $=$ $\lambda_5$ $<$ $ \infty$,\quad
$\displaystyle \frac{\partial f_5}{\partial R}$  $=$ $|-(\lambda_7+\mu_5)|$ $<$ $\infty$,\newline \\\\
Hence we have shown that all the partial derivatives are continuous and bounded.Thus, with the confirmation of the continuity and boundedness of all partial derivatives, the Lipschitz condition is fulfilled. Consequently, a unique solution to the system (\ref{begineq})-(\ref{endeq}) exists within the region D according to theorem \ref{th2}.
\end{proof}

\subsection{Sensitivity Analysis of \texorpdfstring{$ R_0 $}{ }}

Sensitivity analysis is an  important method used to determine the relationship between the model parameters and a disease dynamics. A Partial Rank Correlation Coefficient (PRCC) is a robust sensitivity analysis method that combines ranked correlation and partial correlation to measure the correlation between the input parameter and the output variable. This approach removes any hidden connections between the main parameter and other factors in the model. It also gets rid of any links between the other factors and the final result, letting us concentrate only on how the main parameter affects the outcome. We present the PRCC analysis to study the influence of parameters to basic reproduction number $R_0$. \\

  Similar to method discussed in \cite{marino2008methodology}, employing the concept of  Partial Rank Correlation, about 1000 latin hypercube samples were generated using Python  where all the  parameters were assumed to follow a uniform distribution. The expression for $R_0,$  given by   
\[
\textbf{R}_0 = \frac{\lambda_2 b}{(\lambda_1 + \mu_1)(\lambda_3 +\mu_2)}
\]
is used as output variable to each of the latin hypercube samples. 
\begin{figure}[ht] 
    \centering
    \includegraphics[width=12cm, height=6cm]{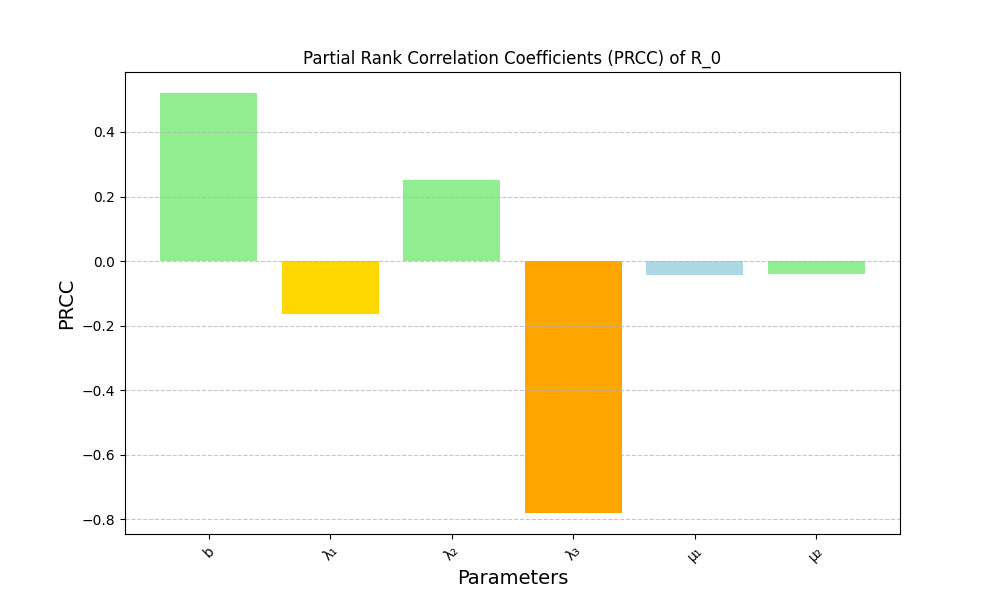} 
    \captionof{figure}{PRCC analysis of $R_0$ } 
    \label{PRCC_R0)}
\end{figure}
\\

The output values of $R_0$ corresponding to  each of the set of parameters from latin hypercube sample were further used to perform PRCC analysis of $R_0$.
 PRCC assigns a parameter a value between -1 to +1 where its magnitude represents the parameter importance  while the signs represents the direction of relationship between the input and output variables. Negative PRCC implies that when a parameter increases the output variable decreases  and when parameter value decreases, the output variable value increases. Positive PRCC means that when parameter value increases output variable  value also increases and output value decreases when parameter decreases. \\
 
 As presented in figure \ref{PRCC_R0)}, parameters that is positively correlated to $R_0$ are $\lambda_2$ and b which are transmission rate to tuberculosis latent compartment from susceptible compartment and natural birth rate. Thus, the most influential parameter is  $\lambda_2$ marked by highest positive PRCC value such that  $R_0$ increases when $\lambda_2$ increases  .\\

 Parameters that are negatively correlated are transmission rate from uninfected  to diseased compartment $\lambda_1$, transmission rate from latent to diseased compartment $\lambda_3$, death rate of uninfected compartment $\mu_1$ and death rate of latently infected compartment  $\mu_2$ respectively. The most influential parameter is  $\lambda_3$ with highest PRCC negative value such that $R_0$ decreases when it increases.\\
 
Therefore, it is important that our intervention should target in minimising the parmaeter $\lambda_2$ and maximising the parameters $\lambda_1$ and $\lambda_3$ respectively.

\section{Optimal Control Studies}
  Optimal control strategies consists of  a range of interventions aimed at reducing TB transmission, improving treatment outcomes, and addressing risk factors associated with TB progression. Hunger-related malnutrition significantly increases the susceptibility to TB, raising the risk by 6 to 10 times. Additionally, it serves as a contributing factor to the progression from latent TB infection to active TB \cite{Sinha}. This transition is often marked by the reactivation of latent TB, frequently observed in individuals with a low Body Mass Index (BMI). Studies have shown that the risk of TB escalates by 13.8 \% \cite{Putri} with each unit decrease in BMI. The relationship between TB and malnutrition is two-way. Malnutrition can be caused by TB, due to inflammation-related issues like cachexia, anorexia, and malabsorption \cite{Scrimshaw}. Similarly, Diabetes Mellitus (DM) significantly impacts tuberculosis in several ways, complicating both the management and outcomes of the disease. Firstly, diabetes weakens the immune system, making individuals more susceptible to TB infection. Moreover, diabetes alters the immune response to TB, affecting  body's ability to control the TB bacteria, leading to more severe and prolonged TB infections. Diabetes makes a person’s risk of getting TB 3 times bigger \cite{diabetes}.\\
 
  In the study, we aim to  improve the treatment through intervention like TB Preventive Treatment (TPT) and supplementary interventions for conditions like Diabetes (D) and Malnutrition (MN). Tuberculosis Preventive Treatment involves providing medication to individuals at high risk of developing active TB disease, such as close contacts of TB patients or individuals with latent TB infection (TBI). By treating latent infections, TPT reduces the risk of progression to active TB, thereby decreasing transmission and preventing new cases. Diabetes mellitus (DM) is a known risk factor for TB, as it impairs the immune response and increases susceptibility to infections. Optimal control of TB involves screening for and managing DM among TB patients, ensuring timely diagnosis, treatment, and glycemic control. Integration of TB and DM services facilitates comprehensive care and improves treatment outcomes. Malnutrition weakens the immune system and predisposes individuals to TB infection and disease progression. Optimal control strategies include nutritional supplementation programs aimed at improving the nutritional status of TB patients and at-risk populations. Adequate nutrition supports immune function, enhances treatment response, and reduces TB-related morbidity and mortality. \\
 
 We examine a control problem involving interventions such as TB preventive treatment (TPT) and additional measures for conditions like diabetes (D) and malnutrition (MN). The dynamic model, including control variables, is represented by the following system of nonlinear differential equations.

\numberwithin{equation}{section}
\begin{align}
   \frac{dU}{dt} &= b - (\lambda_1 + \mu_1)U - \lambda_2 T_l U \label{begcost} \\
   \frac{dT_l}{dt} &= \lambda_2 T_l U - (\lambda_3 + \mu_{1T}(t) + \mu_{1M}(t) + \mu_{1D}(t) + \mu_{2})T_l \label{eq:sec11equ21} \\
   \frac{dT_d}{dt} &= \lambda_1 U + \lambda_3 T_l - (\lambda_4 + \mu_{2T}(t) + \mu_{2M}(t) + \mu_{2D}(t) + \mu_3 + \sigma) T_d + \lambda_6 T_t + \lambda_7 R \label{eq:sec11equ33} \\
   \frac{dT_t}{dt} &= \lambda_4 T_d - (\lambda_5 + \lambda_6 + \mu_{3T}(t) + \mu_{3M}(t) + \mu_{3D}(t) + \mu_4) T_t \label{eq:sec11equ41} \\
   \frac{dR}{dt} &= \lambda_5 T_t - (\lambda_7 + \mu_5) R \label{endcost}
\end{align}

 For simplicity, we define $U_{T}$,$U_{M}$ and\hspace{0.1em} $U_{D}$ as follows,\newline
\indent \hspace{10em} $U_{T}={(\mu_{1T}, \mu_{2T}, \mu_{3T})}$,\quad $U_{M}={(\mu_{1M},\mu_{2M}, \mu_{3M})}$\newline
\indent \hspace{10em} $U_{D}={(\mu_{1D},\mu_{2D}, \mu_{3D})}$

 With this notation, the set of all admissible controls is given by,
\[
\begin{aligned}
&U_c = \Bigg\{ (U_{T}(t),U_{M}(t), U_{D}(t)) : 
U_{T}(t) \in [0,U_{T}\text{max}],\\
&\hspace{13em} U_{M}(t) \in [0,U_{M}\text{max}], \\
&\hspace{13em} U_{D}(t) \in [0,U_{D}\text{max}] \Bigg\}
\end{aligned}
\]
 Here, all the control variables are measurable and bounded functions, and T is the final time of the applied control interventions. The upper bounds of control variables are based on the resource limitation and the limit to which these intervention are assessed. Our main objective of this study is to investigate such optimal control functions that maximizes the benefits of each of the drug interventions and minimize infection.\newline 
 Based on the above assumptions, we wish to minimize the objective cost functional given by
\large

\begin{equation} \label{cost functional}
\begin{split}
 \textit{J($U_{T}$, $U_{M}$, $U_{D}$)} = \int_{0}^{T} & \Big[ A1\Big({\mu_{1T}^2} +{\mu_{2T}^2} +{\mu_{3T}^2}\Big) 
 + A2\Big({\mu_{1M}^2} +{\mu_{2M}^2} +{\mu_{3M}^2}\Big)\\
 &+ A3\Big({\mu_{1D}^2} +{\mu_{2D}^2} +{\mu_{3D}^2}\Big) 
 +T_l(t)+T_d(t)+T_t(t)\Big]dt
\end{split}
\end{equation}

subject to the system,
\normalsize
\numberwithin{equation}{section}
\begin{align}
   \frac{dU}{dt} &= b - (\lambda_1 + \mu_1)U - \lambda_2 T_l U \label{} \\
   \frac{dT_l}{dt} &= \lambda_2 T_l U - (\lambda_3 + \mu_{1T}(t) + \mu_{1M}(t) + \mu_{1D}(t) + \mu_{2})T_l \label{eq:sec11equ22} \\
   \frac{dT_d}{dt} &= \lambda_1 U + \lambda_3 T_l - (\lambda_4 + \mu_{2T}(t) + \mu_{2M}(t) + \mu_{2D}(t) + \mu_3 + \sigma) T_d + \lambda_6 T_t + \lambda_7 R \label{eq:sec11equ34} \\
   \frac{dT_t}{dt} &= \lambda_4 T_d - (\lambda_5 + \lambda_6 + \mu_{3T}(t) + \mu_{3M}(t) + \mu_{3D}(t) + \mu_4) T_t \label{eq:sec11equ43} \\
   \frac{dR}{dt} &= \lambda_5 T_t - (\lambda_7 + \mu_5) R \label{}
\end{align}
 with initial conditions \( U(0) \geq 0 \), \( TBI(0) \geq 0 \), \( TBD(0) \geq 0 \), \( TT(0) \geq 0 \), and \( R(0) \geq 0 \).
 The inclusion of quadratic terms in the definition of the objective function reflects the multiple effects that interventions can have when administered
 \cite{joshi2002optimal}. This approach acknowledges that interventions may not only produce linear effects but also interact in complex ways, leading to quadratic relationships. Furthermore, defining the objective function as a linear combination of quadratic terms of control variables simplifies the problem's complexity. Studies \cite{kamyad2014mathematical, madubueze2020controlling} have demonstrated its efficacy in addressing optimization challenges while still accounting for the interactions between interventions. However, it's worth noting that incorporating higher orders of control variables in objective functions can introduce complications \cite{khatua2020dynamic, lee2010optimal}.\\
 
The integrand of the cost functional (\ref{cost functional}) is given by:
\begin{equation} \label{eq:integrand_cost_functional}
  \begin{aligned}
  \textit{L}(U_{T}, U_{M}, U_{D}, \text{TBI}, \text{TBD}, 
  \text{TT}) &= A1\Big(\mu_{1T}^2 + \mu_{2T}^2 + 
  \mu_{3T}^2\Big) \\
  &\quad + A2\Big(\mu_{1M}^2 + \mu_{2M}^2 + \mu_{3M}^2\Big) \\
  &\quad + A3\Big(\mu_{1D}^2 + \mu_{2D}^2 + \mu_{3D}^2\Big) \\
  &\quad + \text{$T_l$}(t) + \text{$T_d$}(t) + \text{$T_t$}(t)
  \end{aligned}
 \end{equation}
is called as the Lagrangian of the running cost.  \\

 The cost functional (\ref{cost functional}), quantifies the advantages gained from implementing interventions, while also tracking the dynamics of latently infected, diseased, and treated individuals over the observation period. Our objective is dual: to minimize the counts of latent infections, active diseases, and individuals under treatment throughout this period, while also minimizing the extent of intervention required. The coefficients 
 $A_i$, $i = 1,2,3,$ represent positive weight constants linked to the benefits of each intervention. Our aim is  to identify the most effective intervention strategy (represented by($U_{T}$, $U_{M}$, $U_{D}$)) that minimizes tuberculosis burden while optimizing resource allocation and maximizing intervention benefits. \\

 The admissible solution set for the optimal control problem (\ref{begcost})-(\ref{endcost}) is given by
 \[
 \Omega = \left\{ (U_{T},U_{M},U_{D},U, T_l, T_d, T_t, R ) \,\middle|\, U, T_l, T_d, T_t, \text{ and } R \text{ satisfy } (\ref{begcost})-(\ref{endcost}), \forall U_i \in U_c \right\}
 \]
 All the control variables considered here are measurable and bounded functions.
 The upper limits of the control variables depends on the resource constraint.
\subsection{Existence of Optimal Controls}
 It is essential to address a foundational question: does an optimal  solution even exist?  An existence theorem certifies that the problem has a solution. To establish the existence of optimal control functions that minimize the objective function over a finite time interval [0,T], we aim to establish the conditions outlined in theorem  \ref{th4} by Fleming and Rishel \cite{fleming2012deterministic}. The following theorem provides a framework for verifying the existence of optimal control solutions.
 \begin{theorem} \label{th4}
   There exists a 9-tuple of optimal controls,
   \begin{align*}
    \Omega = \{ &\mu_{1T}^*, \mu_{2T}^*, \mu_{3T}^*, \mu_{1M}^*, \mu_{2M}^*, \mu_{3M}^*, 
               \mu_{1D}^*, \mu_{2D}^*, \mu_{3D}^* \}
   \end{align*}
   in the set of admissible controls $U_c$ such that the cost functional (5.6) is minimized corresponding to the optimal control problem (\ref{begcost})-(\ref{endcost}).
\end{theorem}

\begin{proof}
 In order to show the existence of optimal control functions, we will show that the following conditions are satisfied:\\
 
 1. The solution set for the system (5.1)-(5.5) along with bounded controls must be non-empty, i.e.,\(\Omega \neq \emptyset\). \\
 
 2. Control set \( U_c \) is closed and convex, and the system should be expressed linearly in terms of the control variables with coefficients that are functions of time and state variables.\\
 
 3. The Lagrangian, i.e., \( L \), is convex on \( U_c \)  and \( L(U_{T}, U_{M}, U_{D}, TBI, TBD, TT) \geq g(U_{T}, U_{M}, U_{D}) \) is a continuous function of control variables such that \( \left| (U_{T}, U_{M}, U_{D}) \right|^{-1} g(U_{T}, U_{M}, U_{D}) \rightarrow \infty \) whenever \( \left| (U_{T}, U_{M}, U_{D}) \right| \rightarrow \infty \), where \( \left| \hspace{1em}\right| \) is the \( l^2 (0, T) \) norm. \\

we will show that each of the conditions are satisfied:\\

 1. From Positivity and boundedness of solutions of the system (\ref{begineq})-(\ref{endeq}), all solutions are bounded for each bounded control variable in \( U\). Also, the right-hand side of the system (\ref{begineq})-(\ref{endeq}) satisfies Lipschitz condition with respect to state variables. Hence, using the positivity and boundedness condition and the existence of solution from Picard-Lindelof Theorem \cite{makarov2013picard}, we have satisfied condition 1.\\
 
 2. \( U_c \) is closed and convex by definition. Also, the system  (\ref{begcost})-(\ref{endcost}) is clearly linear with respect to controls such that coefficients are only state variables or functions dependent on time. Hence condition 2 is satisfied.\\
 
 3. Choosing \( g(U_{T},U_{M},U_{D}) = \kappa \left( \mu_{1T}^2 + \mu_{2T}^2 + \mu_{3T}^2 + \mu_{1M}^2 + \mu_{2M}^2 + \mu_{3M}^2 + \mu_{1D}^2 + \mu_{2D}^2 + \mu_{3D}^2 \right) \) where \( \kappa = \min\{A1,A2,A3\} \), condition 3 is satisfied. \\

Hence there exists a control 9-tuple, \( (\mu_{1T}^*,\mu_{2T}^*,\mu_{3T}^*,\mu_{1M}^*,\mu_{2M}^*,\mu_{3M}^*,\mu_{1D}^*,\mu_{2D}^*,\mu_{3D}^*) \in U_c \) that minimizes the cost function \ref{cost functional}.
\end{proof}
\subsection{Characterization of the Optimal Control}
 It is crucial to derive the essential conditions for optimal control functions. The characterization of optimal control employs a Pontryagin's Maximum Principle as detailed in Liberzon's work \cite{liberzon2011calculus} from 2011.The Maximum Principle, a fundamental concept in the theory of optimal control that involves the Hamiltonian.  \\
 The Hamiltonian for this problem is given by
\begin{align*}
    &H(U, T_l, T_d, T_t, R, \lambda, \mu_{1T}^*, \mu_{2T}^*, \mu_{3T}^*, \mu_{1M}^*, \mu_{2M}^*, \mu_{3M}^*, \mu_{1D}^*, \mu_{2D}^*, \mu_{3D}^*) \\
    &\quad = L(T_l, T_d, T_t, \mu_{1T}^*, \mu_{2T}^*, \mu_{3T}^*, \mu_{1M}^*, \mu_{2M}^*, \mu_{3M}^*, \mu_{1D}^*, \mu_{2D}^*, \mu_{3D}^*) \\
    &\qquad + \lambda_{U} \frac{dU}{dt} + \lambda_{T_l} \frac{dT_l}{dt} + \lambda_{T_d} \frac{dT_d}{dt} + \lambda_{T_t} \frac{dT_t}{dt} + \lambda_{R} \frac{dR}{dt}
 \end{align*}
  where, $\lambda_c=(\lambda_U, \lambda_{T_l}, \lambda_{T_d}, \lambda_{T_t}, \lambda_R) $ is called as the co-state vector or the adjoint vector and $\lambda_U(T)=0, \lambda_{T_l}(T)=0, \lambda_{T_d}(T)=0,  \lambda_{T_t}(T)=0, \lambda_R(T)=0.$
  Now the canonical equations that relate the state variables to the co-state variables are given by 

 \begin{align}
 \frac{d\lambda_U}{dt} &= -\frac{\partial H}{\partial U} \label{eq:lambda_U} \\
 \frac{d\lambda_{T_l}}{dt} &= -\frac{\partial H}{\partial T_l} \label{eq:lambda_TBI} \\
 \frac{d\lambda_{T_d}}{dt} &= -\frac{\partial H}{\partial T_d} \label{eq:lambda_TBD} \\
 \frac{d\lambda_{T_t}}{dt} &= -\frac{\partial H}{\partial T_t} \label{eq:lambda_TT} \\
 \frac{d\lambda_R}{dt} &= -\frac{\partial H}{\partial R} \label{eq:lambda_R}
 \end{align}

substituting the value of Hamiltonian, we get

\begin{eqnarray}
\frac{d\lambda_U}{dt} &=& \lambda_U (\lambda_1 + \mu_1 + \lambda_2 T_l) - \lambda_{T_l} \lambda_{2} T_l - \lambda_{T_d} \lambda_{1} \label{eq:sec11equ1} \\
\frac{d\lambda_{T_l}}{dt} &=& -\left\{ 1 - \lambda_U \lambda_2 U + \lambda_{T_l} \lambda_{2} U - \lambda_{T_l} (\lambda_3 + \mu_{1T} + \mu_{1M} + \mu_{1D} + \mu_2) + \lambda_{T_d} \lambda_3\right\} \label{eq:sec11equ23} \\
\frac{d\lambda_{T_d}}{dt} &=& -\left\{1 - \lambda_{T_d} (\lambda_4 + \mu_{2T} + \mu_{2M} + \mu_{2D} + \mu_3 + \sigma) +\lambda_{T_t} \lambda_4 \right\} \label{eq:sec11equ35} \\
\frac{d\lambda_{T_t}}{dt} &=& -\left\{1 - \lambda_{T_t} (\lambda_5 + \lambda_6 + \mu_{3T} + \mu_{3M} + \mu_{3D} + \mu_4 ) +\lambda_R \lambda_5 \right\} \label{eq:sec11equ44} \\
\frac{d\lambda_{R}}{dt} &=& \lambda_R (\lambda_7 + \mu_5)-\lambda_{T_d} \lambda_7 \label{eq:sec11equ5}
\end{eqnarray}
along with transversality conditions, 
$\lambda_U(T)=0,\lambda_{T_l}(T)=0,\lambda_{T_d}(T)=0,  \lambda_{T_t}(T)=0, \lambda_R(T)=0.$\\

We will use the Hamiltonian minimization condition to obtain the optimal controls,
\begin{center}
\Large
    \(\frac{\partial H}{\partial u_i} = 0\) at \(u_i = u_i^*\), \small where $u_i$ is any component in the the 9-tuple of optimal control.
\end{center}
Differentiating the Hamiltonian and solving the equations,
we obtain the optimal controls as\newline 
\indent\hspace{10em}$\mu_{1T}^*=\min\{\hspace{0.3 em} \max\{ \frac{\lambda_{T_l}\hspace{0.3 em} T_l}{2A1},\hspace{0.3 em}0\},\hspace{0.3 em} \mu_{1T}max\}$ \newline 
\indent\hspace{10em}$\mu_{2T}^*=\min\{\hspace{0.3 em} \max\{ \frac{\lambda_{T_d}\hspace{0.3 em} T_d}{2A1},\hspace{0.3 em}0\},\hspace{0.3 em} \mu_{2T}max\}$ \newline 
\indent\hspace{10em}$\mu_{3T}^*=\min\{\hspace{0.3 em} \max\{ \frac{\lambda_{T_t}\hspace{0.3 em} T_t}{2A1},\hspace{0.3 em}0\},\hspace{0.3 em} \mu_{3T}max\}$ \newline 
\indent\hspace{10em}$\mu_{1M}^*=\min\{\hspace{0.3 em} \max\{ \frac{\lambda_{T_l}\hspace{0.3 em} T_l}{2A2},\hspace{0.3 em}0\},\hspace{0.3 em} \mu_{1M}max\}$ \newline 
\indent\hspace{10em}$\mu_{2M}^*=\min\{\hspace{0.3 em} \max\{ \frac{\lambda_{T_d}\hspace{0.3 em} T_d}{2A2},\hspace{0.3 em}0\},\hspace{0.3 em} \mu_{2M}max\}$ \newline 
\indent\hspace{10em}$\mu_{3M}^*=\min\{\hspace{0.3 em} \max\{ \frac{\lambda_{T_t}\hspace{0.3 em} T_t}{2A2},\hspace{0.3 em}0\},\hspace{0.3 em} \mu_{3M}max\}$ \newline 
\indent\hspace{10em}$\mu_{1D}^*=\min\{\hspace{0.3 em} \max\{ \frac{\lambda_{T_l}\hspace{0.3 em} T_l}{2A3},\hspace{0.3 em}0\},\hspace{0.3 em} \mu_{1D}max\}$ \newline 
\indent\hspace{10 em}$\mu_{2D}^*=\min\{\hspace{0.3 em} \max\{ \frac{\lambda_{T_d}\hspace{0.3 em} T_d}{2A3},\hspace{0.3 em}0\},\hspace{0.3 em} \mu_{2D}max\}$ \newline 
\indent\hspace{10 em}$\mu_{3D}^*=\min\{\hspace{0.3 em} \max\{ \frac{\lambda_{T_t}\hspace{0.3 em} T_t}{2A3},\hspace{0.3 em}0\},\hspace{0.3 em} \mu_{3D}max\}$ 

\subsection{Numerical Simulations and Studies}
 A numerical simulations is a method to understand the effectiveness of various interventions incorporated in the model. Studying the scenarios via simulations allows to pinpoint the combination of interventions that delivers the best outcomes. This systematic approach enables us to thoroughly evaluate various treatment strategies and provide well-informed recommendations for clinical practice. We evaluate the efficacy of various combinations of controls as follows: 
\begin{enumerate}
    \item Single intervention
    \item Multiple intervention
\end{enumerate}
 Our goal is to propose the most optimal intervention by examining how different controls impact the dynamics of the system described by equations (\ref{begcost})-(\ref{endcost}).
 The parameter values utilized for simulation are sourced from existing literature on TB and are detailed in table \ref{parameters value}.\\

 To begin, the state system using the fourth-order Runge-Kutta method in Python is solved numerically. The initial values of the state variables \cite{TBREPORT}  are set as follows: U(0)=75570, $T_l$(0)=3443, $T_d$(0)=2310, $T_t$(0)=1892 and R(0)=1608.2, each being scaled by a factor of $10^4$ to represent the population in each compartment. The total duration is set to be sixty months (T = 60) or five years. The initial values of the control parameters are all set to zero. With these initial values, we simulate the system with controls and examine the effects of intervention on reducing infected compartments.\\

 To simulate the system with controls, the Forward-Backward Sweep method is employed. At first, the controls are initialized to zero and state system are solved forward in time. Subsequently, the transversality constraints are addressed by solving the adjoint state system backward in time. The optimal state variables and the initial values of the optimal controls, which are also set to zero are utilized.\\
 
 The values of the adjoint state variables are utilized to update the optimal controls iteratively. This process is repeated with the updated control variables until the convergence criterion, as outlined in Liberzon \cite{liberzon2011calculus}, is met. Additionally,  assigning weights to the objective function in accordance with the paper \cite{silvanonreview,silvareview} and is based as A1 = 55, A2 = 30, A3 = 100.

\begin{table}[htpb]
\centering
\resizebox{\textwidth}{!}{
\begin{tabular}{|c|p{5cm}|p{5cm}|}
\hline
\textbf{Parameter} & \textbf{Value} & \textbf{References} \\
\hline
$\lambda_1$ & 0.083 & \cite{silvanonreview} \\
\hline
$\lambda_2$ & 0.0053 & \cite{nkamba2021stability} \\
\hline
$\lambda_3$ & 0.2 & \cite{mayowa, paularodrigues2014cost} \\
\hline
$\lambda_4$ & 0.241 & \cite{nematollahi2020nonlinear} \\
\hline
$\lambda_5$ & 0.10 & \cite{DanyPascal, Erick, silvanonreview} \\
\hline
$\lambda_6$ & 0.0891 & \cite{Sudipta, pardeshi2010time} \\
\hline
$\lambda_7$ & 0.0003 & \cite{malik2018mathematical, mandal2017counting, Erick} \\
\hline
$b$ & 304.17 & \cite{DanyPascal, gao2018optimal} \\
\hline
$\sigma$ & 0.013 & \cite{Erick} \\
\hline
$\mu_1$ & 0.0008 & \cite{nkamba2021stability} \\
\hline
$\mu_2$ & 0.001 & \cite{nkamba2021stability} \\
\hline
$\mu_3$ & 0.001 & \cite{nkamba2021stability} \\
\hline
$\mu_4$ & 0.001 & \cite{nkamba2021stability} \\
\hline
$\mu_5$ & 0.0008 & \cite{nkamba2021stability} \\
\hline
$A1$ & 55 & \cite{rabiu2021optimal} \\
\hline
$A2$ & 30 & \cite{rabiu2021optimal} \\
\hline
$A3$ & 100 & \cite{Erick, sweilam2019optimal} \\
\hline
N(0) & $1.1 \times 10^9$ & \cite{wiki:demographics-india} \\
\hline
U(0) & $7557 \times 10^5$ & \cite{TBREPORT} \\
\hline
$T_l(0)$ & $3443 \times 10^5$ & \cite{TBREPORT} \\
\hline
$T_d(0)$ & $2310 \times 10^4$ & \cite{TBREPORT} \\
\hline
$T_t(0)$ & $1892 \times 10^3$ & \cite{TBREPORT} \\
\hline
$R(0)$& $1608.2 \times 10^2$ & \cite{TBREPORT} \\
\hline
\end{tabular}
}
\caption{Parameter values used in the simulation}
\label{parameters value}
\end{table}

\subsubsection{Without any Interventions}
  A simulation without any intervention serves as the benchmark to compare the effects of the various interventions in the different compartments. Simulation of the model in the absence of intervention over a time period of 5 years gives an insights to the actual dynamics of the different compartments.  As can be seen from figure \ref{ dynamics without Intervention}, the populations of latent TB infection ($T_l$), active TB disease ($T_d$), individuals undergoing treatment ($T_t$), and those who have completed treatment (R) are all increasing over time. This indicates the progression of individuals through various stages of TB infection and treatment. Meanwhile, the population of uninfected individuals (U) is decreasing, reflecting the spread of TB within the population. Having seen the dynamics of each of the compartments, we are ready to  perform further analysis of our model with different interventions . 
\begin{figure}[ht]
    \centering
    \includegraphics[width=18cm, height=6cm]{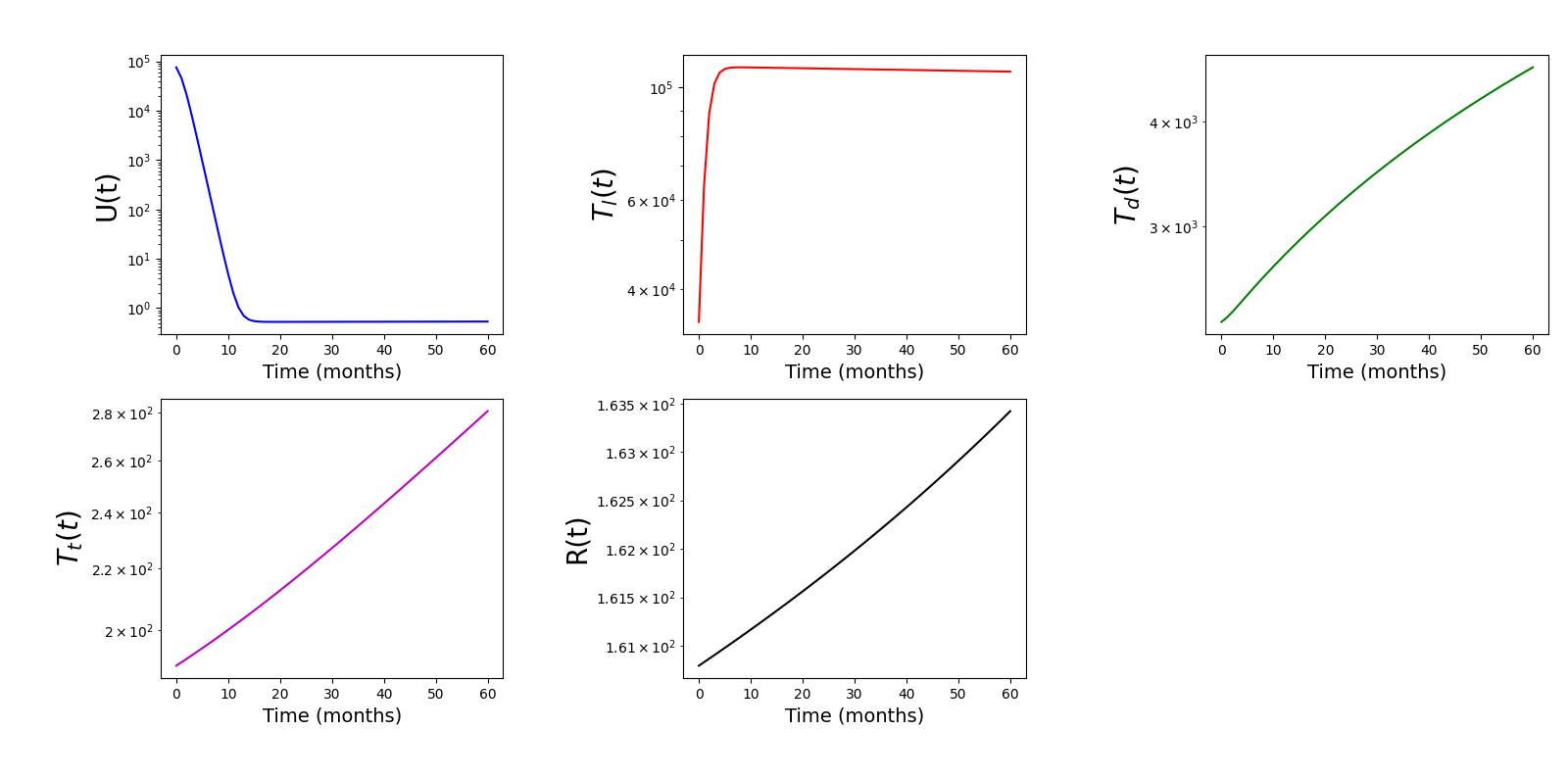}
    \caption{Population dynamics without Intervention 
    }
    \label{ dynamics without Intervention}
\end{figure}

\subsubsection{Single Intervention}
   A simulation spanning over 60 months to examine the behavior of different compartments in tuberculosis dynamics under various interventions applied individually is performed. The compartments include the uninfected population ($U(t)$), latent TB infection ($T_l$), active TB disease  ($T_d$), individuals undergoing treatment ($T_t$), and those who have completed treatment (R). The interventions assessed are Tuberculosis Preventive Treatment (TPT),  Malnutrition (MN), and  Diabetes management (D). \\
   
   Figure \ref{Population dynamics with single intervention} illustrates the temporal evolution of these compartments with each intervention applied separately. Through this analysis, we aim to discern the impact of each intervention on TB dynamics, aiding in the evaluation of their efficacy in mitigating the spread and burden of tuberculosis within the population. Interventions aimed at reducing TB infection rates show an increase in the uninfected population, indicating their preventive efficacy.  Furthermore, the dynamics of the recovered population ($R(t)$) are closely linked to the infected compartments, showing a decrease corresponding to the decrease in the infected population when interventions are assessed. The infected compartments namely $T_l(t)$, $T_d(t)$, and $T_t(t)$ each show a significant decrease in population with the implementation of interventions. Remarkably, the intervention targeting malnutrition (MN) stands out for its significant reduction in the infected population $(T_l, T_d, T_t)$, suggesting its potential effectiveness in curbing the spread and burden of tuberculosis within the population.
\begin{figure}[ht]
    \centering
    \includegraphics[width=18cm, height=6cm]{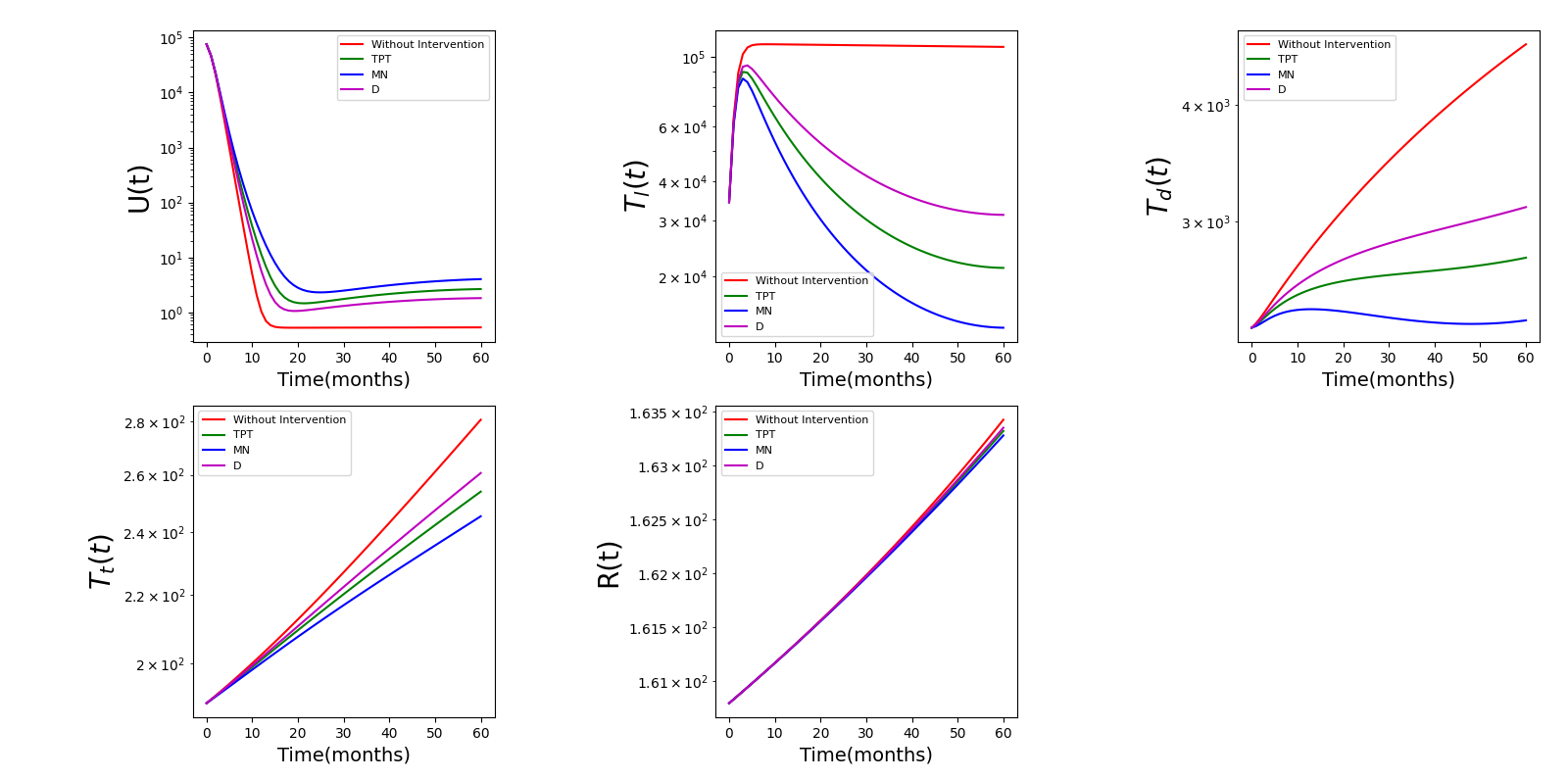}\\
    \caption{Population dynamics with single intervention over the time period of 5 years.}
    \label{Population dynamics with single intervention}
\end{figure}  

\subsubsection{Multiple Intervention}
  A simulation to explore the behavior of various compartments in tuberculosis dynamics under different combinations of interventions is performed. As depicted in figure \ref{Population dynamics with multiple intervention }, the combination of Tuberculosis Preventive Treatment (TPT) and addressing malnutrition (MN) notably reduces the infected and recovered populations while increasing the uninfected population to a greater extent compared to other intervention combinations involving two interventions. Consequently, this combination emerges as the most effective. Furthermore, when all the three interventions are administered simultaneously, they collectively prove to be the most effective in decreasing the infected compartments and boosting the uninfected population. Thus, the combination of all three interventions is deemed the optimal approach for treating patients undergoing tuberculosis treatment.
\begin{figure}[ht]
    \centering
    \includegraphics[width=18cm, height=6cm]{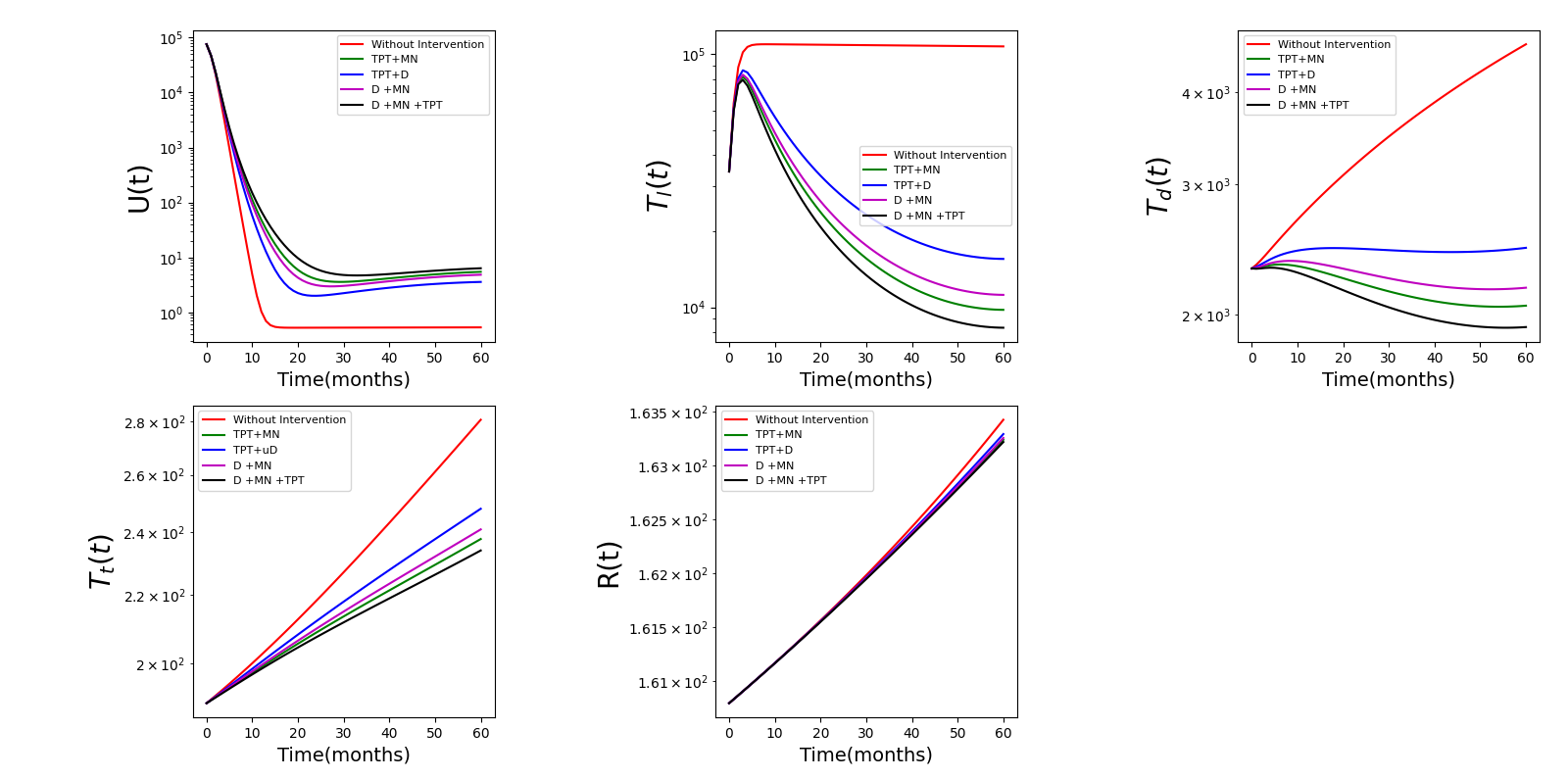}\\
     \caption{Population dynamics with multiple intervention over the time period of 5 years.}
     \label{Population dynamics with multiple intervention }
\end{figure}

\subsection{Optimal Control Strategy}
  The characterization of an optimal control resulted in the following expression 

\indent\hspace{10em}$\mu_{1T}^*=\min\{\hspace{0.3 em} \max\{ \frac{\lambda_{T_l}\hspace{0.3 em} T_l}{2A1},\hspace{0.3 em}0\},\hspace{0.3 em} \mu_{1T}max\}$ \newline 
\indent\hspace{10em}$\mu_{2T}^*=\min\{\hspace{0.3 em} \max\{ \frac{\lambda_{T_d}\hspace{0.3 em} T_d}{2A1},\hspace{0.3 em}0\},\hspace{0.3 em} \mu_{2T}max\}$ \newline 
\indent\hspace{10em}$\mu_{3T}^*=\min\{\hspace{0.3 em} \max\{ \frac{\lambda_{T_t}\hspace{0.3 em} T_t}{2A1},\hspace{0.3 em}0\},\hspace{0.3 em} \mu_{3T}max\}$ \newline 
\indent\hspace{10em}$\mu_{1M}^*=\min\{\hspace{0.3 em} \max\{ \frac{\lambda_{T_l}\hspace{0.3 em} T_l}{2A2},\hspace{0.3 em}0\},\hspace{0.3 em} \mu_{1M}max\}$ \newline 
\indent\hspace{10em}$\mu_{2M}^*=\min\{\hspace{0.3 em} \max\{ \frac{\lambda_{T_d}\hspace{0.3 em} T_d}{2A2},\hspace{0.3 em}0\},\hspace{0.3 em} \mu_{2M}max\}$ \newline 
\indent\hspace{10em}$\mu_{3M}^*=\min\{\hspace{0.3 em} \max\{ \frac{\lambda_{T_t}\hspace{0.3 em} T_t}{2A2},\hspace{0.3 em}0\},\hspace{0.3 em} \mu_{3M}max\}$ \newline 
\indent\hspace{10em}$\mu_{1D}^*=\min\{\hspace{0.3 em} \max\{ \frac{\lambda_{T_l}\hspace{0.3 em} T_l}{2A3},\hspace{0.3 em}0\},\hspace{0.3 em} \mu_{1D}max\}$ \newline 
\indent\hspace{10 em}$\mu_{2D}^*=\min\{\hspace{0.3 em} \max\{ \frac{\lambda_{T_d}\hspace{0.3 em} T_d}{2A3},\hspace{0.3 em}0\},\hspace{0.3 em} \mu_{2D}max\}$ \newline 
\indent\hspace{10 em}$\mu_{3D}^*=\min\{\hspace{0.3 em} \max\{ \frac{\lambda_{T_t}\hspace{0.3 em} T_t}{2A3},\hspace{0.3 em}0\},\hspace{0.3 em} \mu_{3D}max\}
$\\
 From the expression itself, it is difficult to understand how the actual optimal strategy would look like. Therefore, simulation of the trajectory of optimal controls would give a better insights. The equality \\
 \indent\hspace{6em}$\mu_{T}^*$= $\mu_{1T}^*$ $+$ $\mu_{2T}^*$ $+$ $\mu_{3T}^*$, represents the expression for single intervention TPT.\\
 Similarly,\\
 \indent\hspace{6em} $\mu_{T}^*$ $+$ $\mu_{M}^*$= $\mu_{1T}^*$ $+$ $\mu_{2T}^*$ $+$ $\mu_{3T}^*$ $+$ $\mu_{1M}^*$ $+$ $\mu_{2M}^*$ $+$ $\mu_{3M}^*$, represents the expression for multiple intervention TPT and MN .\\
 \indent\hspace{6em} $\mu_{T}^*$ $+$ $\mu_{M}^*$ $+$ = $\mu_{1T}^*$ $+$ $\mu_{2T}^*$ $+$ $\mu_{3T}^*$  $+$ $\mu_{1M}^*$ $+$ $\mu_{2M}^*$ $+$ $\mu_{3M}^*$  $+$ $\mu_{1D}^*$ $+$ $\mu_{2D}^*$ $+$ $\mu_{3D}^*$, represents the expression for multiple intervention TPT, MN and D. In the same manner, the expression for various combination of interventions is obtained.\\

\begin{figure}[ht]
    \centering
    \begin{minipage}[b]{0.48\textwidth}
        \centering
        \includegraphics[width=\textwidth]{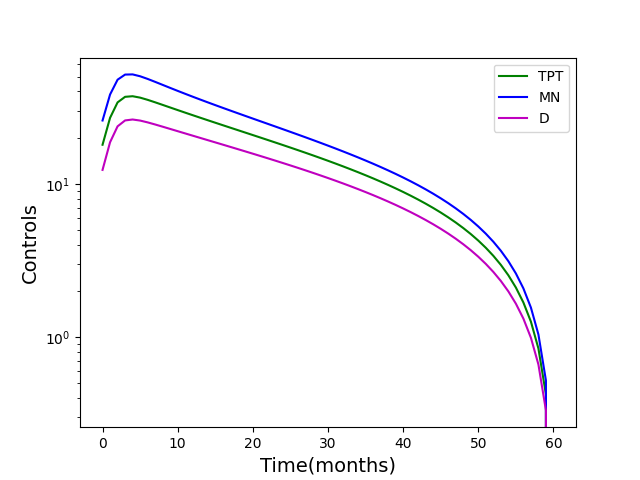}
        \caption{Single intervention }
        \label{Strategy 1}
    \end{minipage}
    \hfill
    \begin{minipage}[b]{0.48\textwidth}
        \centering
        \includegraphics[width=\textwidth]{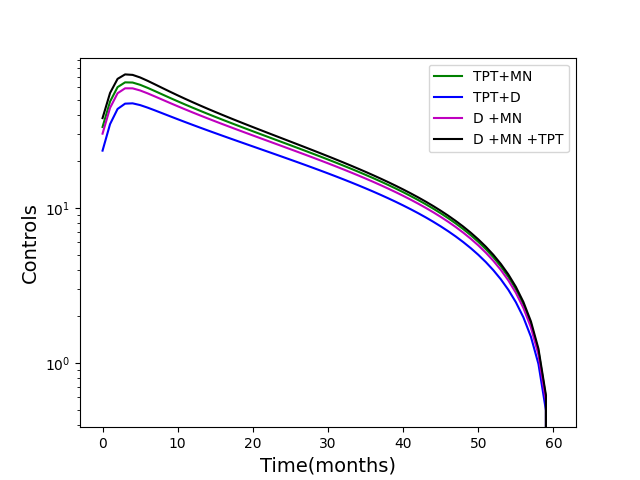}
        \caption{multiple intervention }
          \label{Strategy 2}
    \end{minipage}
\end{figure}

  It can be seen from figure \ref{Strategy 1} and  \ref{Strategy 2}, optimal control strategy for treating tuberculosis demonstrates an initial peak in effectiveness, with the control variable reaching its maximum value at the onset of treatment. This aggressive approach allows for swift and intensive intervention, aiming to rapidly suppress the spread of the disease and mitigate its impact on the population. However, as treatment progresses and the disease is brought under control, the optimal control gradually decreases over time, eventually reaching zero by the end of the intervention period. This tapering-off reflects a shift in focus from acute management to maintenance and consolidation, aligning with the natural course of treatment and the decreasing urgency of the disease burden. By dynamically adjusting the treatment intensity in response to the evolving needs of the patient population, our optimal control strategy aims to optimize treatment outcomes while minimizing unnecessary medication exposure and potential side effects.

\subsection{Population Dynamics }
  Analyzing the average population distribution across different compartments representing various stages of TB progression and treatment response over time for each intervention provide an insights about the effectiveness of interventions.  This comparison helps to understand as to which interventions are most successful in combating tuberculosis.\\

  Table \ref{tab:single_population} and table \ref{tab:multiple_population}  presents the population dynamics of tuberculosis over a 5 year period with single and multiple interventions.  Notably, the combined TPT, MN, and D intervention stands out for its effectiveness in preventing TB transmission, as evidenced by the highest population size among uninfected individuals (U) and considerable reductions in the population of latent TB infection ($T_l$), active TB disease ($T_d$), and individuals undergoing treatment ($T_t$). This suggests a robust impact on TB transmission rates. Additionally, comprehensive strategies exhibit promise in reducing $T_l$ and $T_d$  incidences, underlying the importance of addressing of malnutrition and diabetes in TB management. While treatment rates vary across interventions, recovery rates remain relatively stable ($\approx 164$), indicating the overall efficacy of treatment approaches. \\
  
\begin{table}[htbp]
    \centering
    
    \renewcommand{\arraystretch}{1.5} 
    \begin{tabular}{|c|c|c|c|c|c|c|}
        \hline
        \textbf{Year} & \textbf{Intervention} & {$U$} & \textbf{$T_l$} & \textbf{$T_d$} & \textbf{$T_t$}&  \textbf{$ R$}  \\
        \hline
        \multirow{3}{*} \ {2021} & None & 128.88 & 109478.87 & 2688.50 & 200.03 &161.23 \\
        \cline{2-7}
         & TPT & 182.64 & 64589.28 & 2503.23 & 198.99 &161.23\\
        \cline{2-7}
         & MN & 224.09 & 53964.97 & 2412.27 & 198.23 &161.23\\
        \cline{2-7}
         & D & 160.43 & 74697.17 & 2565.79 & 199.42 &161.23\\
        \hline
        \multirow{3}{*} \ {2022} & None & 0.56 & 109107.82 & 3167.62 & 215.38 &161.28\\
        \cline{2-7}
         & TPT & 1.54 & 38397.30 & 2600.28 & 211.57 &161.28\\
        \cline{2-7}
         & MN & 2.60 & 27907.27 & 2397.41 & 209.41 &161.28\\
        \cline{2-7}
         & D & 1.15 & 50313.20 & 2757.40 & 212.94 &161.28\\
        \hline
        \multirow{3}{*} \ {2023} & None & 0.53 & 108610.73 & 3640.44 & 233.30 &162.23\\
        \cline{2-7}
         & TPT & 1.93 & 27703.70 & 2641.58 & 224.57&162.23 \\
        \cline{2-7}
         & MN & 2.73 & 18760.12 & 2355.05 & 220.58 &162.23\\
        \cline{2-7}
         & D & 1.42 & 38882.60 & 2880.93 & 227.31 &162.23\\
        \hline
        \multirow{3}{*} \ {2024} & None & 0.53 & 108115.91 & 4106.97 & 253.75 &163.23\\
        \cline{2-7}
         & TPT & 2.37 & 22962.99 & 2678.31 & 237.90 &163.23\\
        \cline{2-7}
         & MN & 3.47 & 15001.46 & 2332.52 & 231.84 &163.23\\
        \cline{2-7}
         & D & 1.67 & 33421.05 & 2983.53 & 242.32&163.23 \\
        \hline
        \multirow{3}{*} \ {2025} & None & 0.53 & 107623.37 & 4567.29 & 276.68 &163.29\\
        \cline{2-7}
         & TPT & 2.64 & 21315.92 & 2733.14 & 251.65&163.29 \\
        \cline{2-7}
         & MN & 3.98 & 13745.99 & 2346.17 & 243.44 &163.29\\
        \cline{2-7}
         & D & 1.81 & 31441.81 & 3090.05 & 257.97 &163.29\\
        \hline
    \end{tabular}
    \caption{Average Population in each compartment with Single intervention}
    \label{tab:single_population}
\end{table}

\begin{table}[htbp]
    \centering

    \renewcommand{\arraystretch}{1.2} 
    \begin{tabular}{|c|c|p{2.5cm}|p{2.5cm}|p{2.5cm}|p{2.5cm}|p{2.5cm}|}
        \hline
        \textbf{Year} & \textbf{variables} & \textbf{No Intervention} & \textbf{TPT and MN} & \textbf{TPT and D} & \textbf{MN and D} & \textbf{TPT, MN, and D} \\
        \hline
        \multirow{5}{*} \ 2021  & U & 128.88 & 275.85 & 209.85 & 252.27 & 317.87 \\
        \cline{2-7}
        & $T_l$ & 109478.87 & 46465.91 & 56907.89 & 49413.81 & 42504.27 \\
        \cline{2-7}
        & $T_d$ & 2688.50 & 2325.52 & 2440.70 & 2362.33 & 2279.57 \\
        \cline{2-7}
        & $T_t$ & 200.03 & 197.38 & 198.48 & 197.76 & 196.92 \\
        \cline{2-7}
        & R & 161.17 & 161.17 & 161.17 & 161.17 & 161.17 \\
        \hline
        \multirow{5}{*} \ 2022   & U & 0.56 & 4.98 & 2.17 & 3.72 & 7.87 \\
        \cline{2-7}
        & $T_l$& 109107.82 & 21628.21 & 30618.20 & 23991.98 & 18825.38 \\
        \cline{2-7}
        & $T_d$ & 3167.62 & 2221.78 & 2458.37 & 2294.58 & 2134.65 \\
        \cline{2-7}
        & $T_t$ & 215.38 & 207.13 & 210.11 & 208.12 & 205.93 \\
        \cline{2-7}
        & R & 161.64 & 161.63 & 161.64 & 161.63 & 161.63 \\
        \hline
        \multirow{5}{*}  \ 2023  & U & 0.53 & 3.76 & 2.47 & 3.27 & 4.73 \\
        \cline{2-7}
        & $T_l$ & 108610.73 & 13833.51 & 20987.84 & 15649.23 & 11840.95 \\
        \cline{2-7}
        & $T_d$ & 3640.44 & 2121.70 & 2439.00 & 2217.00 & 2011.33 \\
        \cline{2-7}
        & $T_t$ & 233.30 & 216.64 & 221.84 & 218.33 & 214.60 \\
        \cline{2-7}
        & R & 162.16 & 162.11 & 162.13 & 162.12 & 162.11 \\
        \hline
        \multirow{5}{*} \ 2024   & U & 0.53 & 4.63 & 3.12 & 4.12 & 5.43 \\
        \cline{2-7}
        & $T_l$ & 108115.91 & 10795.32 & 16947.31 & 12329.28 & 9187.31 \\
        \cline{2-7}
        & $T_d$ & 4106.97 & 2062.77 & 2432.00 & 2171.82 & 1939.15 \\
        \cline{2-7}
        & $T_t$& 253.75 & 226.16 & 233.71 & 228.57 & 223.29 \\
        \cline{2-7}
        & R & 162.71 & 162.62 & 162.65 & 162.63 & 162.61 \\
        \hline
        \multirow{5}{*} \ 2025   & U & 0.53 & 5.39 & 3.54 & 4.77 & 6.25 \\
        \cline{2-7}
        & $T_l$ & 107623.37 & 9806.60 & 15583.38 & 11237.86 & 8333.06 \\
        \cline{2-7}
        & $T_d$ & 4567.29 & 2053.87 & 2455.89 & 2171.18 & 1922.35 \\
        \cline{2-7}
        & $T_t$ & 276.68 & 236.07 & 245.93 & 239.17 & 232.43 \\
        \cline{2-7}
        & R & 163.32 & 163.15 & 163.20 & 163.16 & 163.13 \\
        \hline
    \end{tabular}
    \caption{Average Population in each compartment with multiple interventions.}
    \label{tab:multiple_population}
\end{table}

\section {Pseudo-Prevalence and Incidence }
Studying pseudo-prevalence and incidence of tuberculosis  at a country level basis involves using mathematical models. The mathematical model (\ref{begcost})-(\ref{endcost}) is used to approximate the  pseudo-prevalence and  incidence where the parameters values  are used from table \ref{parameters value}.
To calculate the pseudo-prevalence and incidence, the following assumptions are made:
\begin{itemize}
 \item The population of 2021 in India above the age of 15 years is considered to be the initial population N(0) .
 \item  N(0) is considered throughout while calculating the pseudo-prevalence and incidence.
\item  pseudo-prevalence is calculated as the ratio of total number of new cases in the compartments $T_l, T_d$  and $T_t$ i.e. ($T_l(t)+T_d(t)+T_t(t)$) at a time t to the population N(0) .
\item Incidence is calculated as number of new cases in the infected compartment ($T_l(t),T_d(t),T_t(t)$) over the period  of one year divided by N(0) * 365.
\end{itemize}

 Under the assumptions made above, simualtion of pointwise pseudo-prevalence over five years from 2021 to 2025 with single and multiple interventions is represented by figure \ref{fig:prev_sing_mul_plot}. Initially, the the pseudo-prevalence  picks up to the maximum value and then gradually drops down with time,   indicating the effectiveness of the control interventions.

 As per the Global TB Report 2023 \cite{incidence}, the incidence is 199 per 100,000 population in 2022 and 210 per 100,000 in 2021 respectively. Similarly, as reported in \cite{chauhan2023prevalence}, the pooled prevalence for India based on the community-based cohort studies was estimated as 41\% irrespective of the risk of acquiring it, while the estimation was 36\% in 2019-2021. The National Prevalence Survey of India (2019-2021) estimated 31\% tuberculosis infection (TBI) burden among individuals above 15 years of age.\\
 
 Calculation of the total number of new cases ($T_l(t)+T_d(t)+T_t(t)$) at a time t and divided it by  N(0) gives the incidence for the  year 2021, focusing on the age group above 15 years. The incidence was found to be 183.5 per 100,000 population for the year 2021. This indicates that for every 100,000 individuals in the population, there were 183.5 new cases of tuberculosis in 2021. \\

\begin{figure}[ht]
    \centering
    \begin{minipage}[b]{0.45\textwidth}
        \centering
        \includegraphics[width=\textwidth, height=6cm]{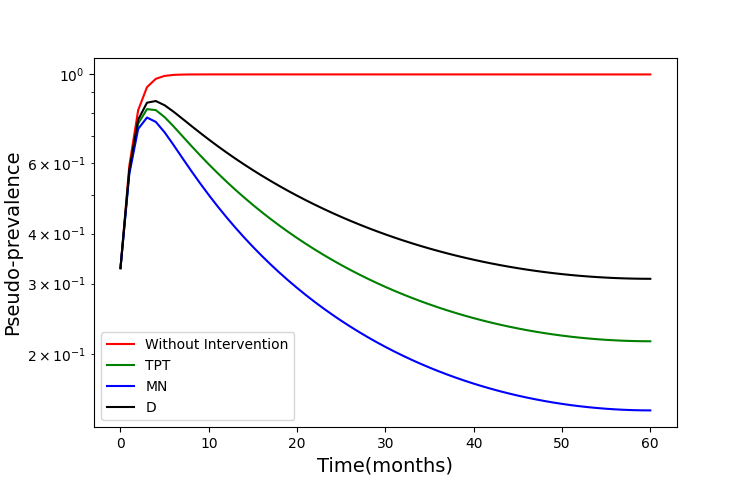} 
        
    \end{minipage}
    \hfill
    \begin{minipage}[b]{0.45\textwidth} 
        \centering
        \includegraphics[width=\textwidth, height=6cm]{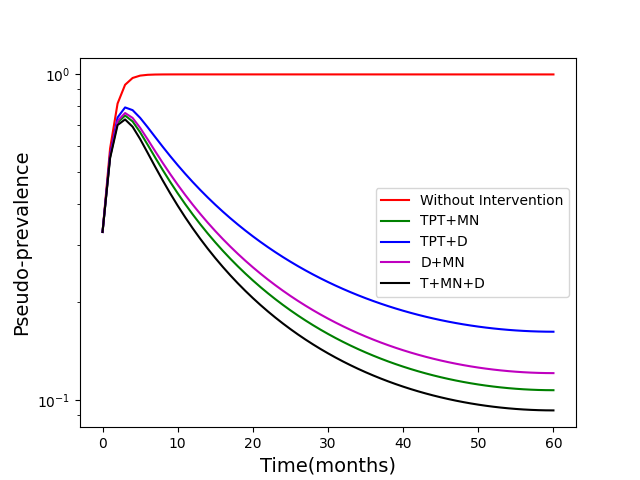}
    \end{minipage}
    \caption{Simulation of pseudo-prevalence with single and multiple interventions over a period of 5 years.}
     \label{fig:prev_sing_mul_plot}
\end{figure}

  Pseudo-prevalence estimates prevalence based on a fixed population size N(0), often assuming stability, whereas actual prevalence accounts for population dynamics N(t), offering insights into how prevalence changes with population fluctuations over time. Indeed, given the similarity between the  pseudo- prevalence and the actual prevalence from table \ref{tab:prevalence with multiple intervention} for the year 2021, leveraging the pseudo-prevalence to project future prevalence appears promising. \\

  Table \ref{tab:prevalence with single intervention} presents the pseudo-prevalence of tuberculosis (TB) across the years 2021 to 2025 with a single intervention strategy. Without any intervention, the prevalence remains relatively high, ranging from 99.8 \% to 99.9\% over the five years. However, the implementation of interventions such as TPT, MN and D significantly reduces the prevalence of TB. MN emerges as the most effective intervention, consistently yielding the lowest prevalence rates, ranging from 46.8\% to 14.5\% across the five years. TPT and D interventions also contribute to lowering prevalence, though to a lesser extent compared to MN.\newline

   Table \ref{tab:prevalence with multiple intervention} provides insights into tuberculosis (TB) prevalence rates with multiple interventions. Notably, a comprehensive strategy combining Tuberculosis Preventive Therapy (TPT), Malnutrition (MN), and Diabetes (D) interventions stands out, consistently yielding the lowest prevalence rates ranging from 9.29\% to 33.63\% over the five-year period.

\begin{table}[htbp]
    \centering
    \renewcommand{\arraystretch}{1.5} 
    \begin{tabular}{|c|c|c|c|c|}
        \hline
        \textbf{Year} & \textbf{No Intervention (\%)} & \textbf{TPT (\%)} & \textbf{MN (\%)} & \textbf{D (\%)} \\
        \hline
        2021 & 99.9 & 56.6 & 46.8 & 66.1 \\ \hline
        2022 & 99.8 & 35.4 & 26.0 & 46.1 \\ \hline
        2023 & 99.8 & 26.6 & 18.5 & 36.7 \\ \hline
        2024 & 99.8 & 22.7 & 15.4 & 32.3 \\ \hline
        2025 & 99.8 & 21.5 & 14.5 & 30.8 \\ 
        \hline
    \end{tabular}
     \caption{Pseudo-prevalence with single intervention}
     \label{tab:prevalence with single intervention}
\end{table}
  
 In summary, the data highlights the superiority of multiple interventions over single intervention in reducing TB prevalence. \\

\begin{table}[htbp]
    \centering
   
    \renewcommand{\arraystretch}{1.5} 
    \begin{tabular}{|c|c|c|c|c|c|}
        \hline
        \textbf{Year} & \textbf{No Intervention (\%)} & \textbf{TPT \& MN (\%)} & \textbf{TPT \& D (\%)} & \textbf{MN \& D (\%)} & \textbf{TPT, MN \& D (\%)} \\
        \hline
        2021 & 99.85 & 37.11 & 46.66 & 39.76 & 33.63 \\ \hline
        2022 & 99.84 & 19.58 & 46.66 & 21.64 & 17.16 \\ \hline
        2023 & 99.82 & 13.68 & 20.13 & 15.33 & 11.89 \\ \hline
        2024 & 99.80 & 11.35 & 17.06 & 12.79 & 9.85 \\ \hline
        2025 & 99.78 & 10.71 & 16.19 & 12.08 & 9.29 \\ 
        \hline
    \end{tabular}
     \caption{Pseudo-prevalence  with multiple interventions}
    \label{tab:prevalence with multiple intervention}
\end{table}
  \section{Cost-Effectiveness Analysis}
  Cost-effectiveness analysis (CEA) serves as a vital tool in identifying interventions or strategies that provide optimal health benefits at minimal costs. It plays a pivotal role in guiding decision-making processes concerning the allocation of resources for public health programs and policies.\\
  
  In our assessment of intervention strategies, we concentrate on three specific population groups: individuals who are latently infected ($T_l$), those with active disease ($T_d$), and those receiving treatment ($T_t$). We evaluate two key aspects: the cost incurred, determined by the area under the cost profile of the intervention strategy, and the health outcomes achieved. Health outcomes are measured as the number of averted cases, which reflects the difference between the number of infected individuals  without any control intervention  and the number of infected individuals  with the implementation of the intervention strategy.\newline
  To quantify the cost incurred, we employ the objective cost functional of our optimal control problem,
 \begin{equation} \label{cost}
 \begin{split}
 \textit{J($U_{T}$, $U_{M}$, $U_{D}$,$T_l$, $T_d$, $T_t$)} = \int_{0}^{T} & \Big[ A1\Big({\mu_{1T}^2} +{\mu_{2T}^2} +{\mu_{3T}^2}\Big) \\
 & + A2\Big({\mu_{1M}^2} +{\mu_{2M}^2} +{\mu_{3M}^2}\Big) + A3\Big({\mu_{1D}^2} +{\mu_{2D}^2} +{\mu_{3D}^2}\Big) \\
 & +T_l(t)+T_d(t)+T_t(t)\Big]dt
 \end{split}
 \end{equation}
 where the  expression in the right hand side of \eqref{cost} represents the cost  associated with all the interventions assesed together at once.
 The last expression in \eqref{cost},
 \begin{equation}  
 \begin{split}
 \textit{J($T_l$, $T_d$, $T_t$)} = \int_{0}^{T} & \Big( T_l(t)+T_d(t)+T_t(t)\Big)dt
 \end{split}
 \end{equation} 
 represents the cost incurred due to disease prevalence.\\
 
 To calculate the cost incurred due to the assessment of TPT alone, we have
 \begin{equation}  \label{costTPT}
 \begin{split}
 \textit{J($U_{T}$,$T_l$, $T_d$, $T_t$)}=\int_{0}^{T} & \Big[ A1\Big({\mu_{1T}^2} +{\mu_{2T}^2} +{\mu_{3T}^2}\Big) +T_l(t)+T_d(t)+T_t(t)\Big]dt
 \end{split}
 \end{equation}
 which is the combination of  cost incurred due to disease prevalence and intervention, TPT. 
 In a similar manner, we  calculate the cost incurred with different combinations of interventions. \\
 
  We have undertaken three different approaches to study cost-effectiveness as described in \cite{agusto2017optimal}: the Average Cost-Effectiveness Ratio (ACER), the Averted Infections Ratio (AIR) and Incremental cost effectiveness ratio (ICER) using the four quadrants of the cost-effectiveness plane. ACER compares the average cost of interventions to the average health outcomes achieved, while AIR measures the ratio of averted infections to the average recovered cases and ICER is an intutive metric to measure the additional cost required to gain  an additional unit of health in comparing two interventions. \\
  
  Table \ref{tab:multiple_interventions} and \ref{tab:single_interventions} presents the total cost incurred and the total number of averted cases with various combination of interventions in 2025.  Cost incurred with intervention D is the highest among all interventions, while intervention MN is associated with the lowest cost. Similarly, the number of averted cases with the intervention TPT, MN, and D together  is greater than any other intervention. In contrast, intervention TPT alone results in the fewest number of averted cases. \\
  
1.$\textbf{ACER (X)} = \frac{\text{Total Cost Incurred by intervention X}}{\text{Total Cases averted by intervention X}}$ \\

  A lower ACER value indicates better cost-effectiveness, meaning that the intervention achieves health outcomes at a lower average cost. Conversely, higher ACER values suggest less cost-effectiveness.\\
  
  As presented in the table \ref{tab:ACER Values for different Interventions}, interventions combining multiple components, such as "TPT and MN" and "MN and D," demonstrate better cost-effectiveness, indicated by lower ACER values. Conversely, interventions targeting single components, like "D," exhibit higher ACER values, suggesting lower cost-effectiveness.\\
 
\begin{table}[ht]
\centering
\begin{tabular}{|c|c|c|c|c|}
\hline
Year & Intervention & Total Cost  & Total Cases Averted \\
\hline

2021 & TPT and MN & 9883267.88 & 32067.47 \\
\cline{2-4}
     & TPT and D & 10254907.52 & 21019.45 \\
\cline{2-4}
     & MN and D & 10005498.34 & 28616.82 \\
\cline{2-4}
     & TPT, MN, and D & 9744866.67 & 35916.55 \\
\hline
\cline{1-1}
2022 & TPT and MN & 17842608.92 & 70067.58 \\
\cline{2-4}
     & TPT and D & 19654368.23 & 56117.19 \\
\cline{2-4}
     & MN and D & 18377039.61 & 66217.41 \\
\cline{2-4}
     & TPT, MN, and D & 17299581.31 & 74116.69 \\
\hline
\cline{1-1}
2023 & TPT and MN & 23478262.76 & 87084.20 \\
\cline{2-4}
     & TPT and D & 26119524.66 & 76221.19 \\
\cline{2-4}
     & MN and D & 24162250.67 & 84245.54 \\
\cline{2-4}
     & TPT, MN, and D & 22878609.26 & 90028.79 \\
\hline
\cline{1-1}
2024 & TPT and MN & 29483936.58 & 95430.95 \\
\cline{2-4}
     & TPT and D & 31597823.35 & 87257.89 \\
\cline{2-4}
     & MN and D & 29863699.65 & 93351.83 \\
\cline{2-4}
     & TPT, MN, and D & 29349132.08 & 97581.45 \\
\hline
\cline{1-1}
2025 & TPT and MN & 37347577.58 & 100105.45  \\
\cline{2-4}
     & TPT and D & 37415260.88 & 93801.88 \\
\cline{2-4}
     & MN and D & 36928638.28 & 98524.16 \\
\cline{2-4}
     & TPT, MN, and D & 38245965.54 & 101743.50 \\
\hline
\end{tabular}
\caption{Total costs and cases avertedwith multiple interventions from 2021 to 2025.}
\label{tab:multiple_interventions}
\end{table}

\begin{table}[htbp]
    \centering
    \begin{minipage}[t]{0.45\textwidth}
        \centering
        \renewcommand{\arraystretch}{1.2} 
        \begin{tabular}{|l|c|}
            \hline
            Intervention & ACER  \\
            \hline
            TPT and MN & 379.66 \\  \hline
            MN and D & 380.46 \\   \hline
            TPT, MN and D & 383.62\\  \hline
            MN & 390.70\\              \hline
            TPT and D & 402.70 \\      \hline
            TPT & 453.71 \\        \hline
            D & 572.67 \\    \hline
        \end{tabular}
        \caption{ACER Values for different Interventions over five years period}
        \label{tab:ACER Values for different Interventions}
    \end{minipage}\hfill
    \begin{minipage}[t]{0.45\textwidth}
        \centering
        \renewcommand{\arraystretch}{1.2} 
        \begin{tabular}{|l|c|}
            \hline
            Intervention & AIR  \\
            \hline
            D & 475.87\\ \hline
            TPT & 540.12 \\ \hline
            TPT and D & 577.01 \\ \hline
            MN & 588.98 \\ \hline
            MN and D & 605.49 \\ \hline
            TPT and MN & 615.04 \\ \hline
            TPT, MN and D & 624.95 \\
            \hline
        \end{tabular}
        \caption{AIR Values for different Interventions over five years period}
        \label{tab:AIR Values for different Interventions}
    \end{minipage}
\end{table}
2.$\textbf{AIR (X)} = \frac{\text{Total Cases Averted by intervention X}}{\text{Total Cases recovered by intervention X}}$\\

 A higher AIR value indicates greater effectiveness in preventing infections relative to the number of cases recovered, while a lower AIR suggests lower effectiveness in averting infections compared to recoveries.\\
 It is clear from  table \ref{tab:AIR Values for different Interventions}, interventions combining multiple components, such as "TPT, MN and D," "TPT and MN," and "MN and D," exhibit higher AIR values, suggesting greater effectiveness in averting infections relative to the number of cases recovered. Conversely, interventions targeting single components, like "D," demonstrate lower AIR values, indicating lower effectiveness in preventing infections compared to the number of cases recovered. \\

 \begin{table}[htbp]
 \centering
 \begin{tabular}{|c|c|c|c|}
 \hline
 Year & Intervention & Tota Cost & Total Cases Averted \\
 \hline
 2021 & TPT & 10439287.17 & 14856.58 \\
  \cline{2-4}   & MN & 10166410.59 & 23806.28 \\
   \cline{2-4}  & D & 10604493.83 & 8877.96 \\
 \hline
 2022 & TPT & 20809428.73 & 45579.35 \\
    \cline{2-4} & MN & 19168838.42 & 60124.14 \\
    \cline{2-4} & D & 22065864.52 & 32103.20 \\
 \hline
 2023 & TPT & 28266438.83 & 66782.44 \\
   \cline{2-4}  & MN & 25323765.89 & 79511.70 \\
   \cline{2-4}  & D & 31041497.18 & 52602.72 \\
 \hline
 2024 & TPT & 34137892.46 & 79635.92 \\
   \cline{2-4}  & MN & 30806969.30 & 89797.01 \\
   \cline{2-4}  & D & 38113476.50 & 67095.43 \\
 \hline
 2025 & TPT & 39565623.05 & 87686.68 \\
   \cline{2-4}  & MN & 36996362.19 & 95786.54 \\
    \cline{2-4} & D & 44051055.14 & 77056.52 \\
 \hline
 \end{tabular}
 \caption{Total costs and cases averted with single interventions from 2021 to 2025.}
 \label{tab:single_interventions}
 \end{table}

3.\textbf{ICER}: Incremental cost effectiveness ratio is a method to measure the economic value of a treatment when compared with the other treatments. It is calculated as a ratio of cost difference between two interventions to the difference in the number of averted cases . An intervention is arranged in an increasing order of their averted cases as shown in \ref{tab:sorted averted cases} and the two competing treatments in terms of averted cases are used in calculating the ICER values as : \\

ICER(D) =  Cost of D/Averted cases with D\\
ICER(TPT) = (Cost of TPT – Cost of D) /(Averted cases with TPT - Averted cases with D)\\
ICER(TPT and D) = (Cost of TPT and D – Cost of TPT) /(Averted cases with TPT and D - Averted cases with TPT) and so on.\\

\begin{table}[htbp]
\centering

\begin{tabular}{|c|c|c|}
\hline
\textbf{Intervention} & \textbf{Cost} & \textbf{Averted Cases} \\\hline
D                     & 44051055.14   & 77056.52 \\\hline
TPT                   & 39565623.05   & 87686.68     \\\hline
TPT and D             & 37415260.88   & 93801.88                \\\hline
MN                    & 36996362.19   & 95786.54                \\\hline

MN and D              & 36928638.28   & 98524.16                \\\hline

TPT and MN            & 37347577.58   & 100105.45               \\\hline

TPT, MN, and D        & 38245965.54   & 101743.50               \\\hline
\end{tabular}
\caption{Costs and Cases averted with different interventions in 2025 in ascending order of Averted cases. }
\label{tab:sorted averted cases}
\end{table}

 \textbf{Four Quadrants Of Cost Effective Analysis (CEA) Plane}: A CEA plane consists of four quadrants i.e. North East (NE), North West (NW), South West (SW), and South East (SE) quadrants. It is a tool to visualise the results of the intervention assessed. It is a graph where the calculated cost of each intervention with the averted cases is plotted and ICER is the metric used to report the results of CEA plane. Averted cases can have general units like quality-adjusted life year (QALY) gained) or units averted. Cost-effectiveness frontier (blue dotted lines) in figure \ref{ThresholdPlot} is a line from one intervention to the next in the increasing order of averted cases. A Cost Effective Threshold (CET) is a red dotted which represents the maximum value one is willing to pay for the additional one unit health outcome (QALY). The cost-effective threshold serves as a benchmark or reference point. If the cost-effectiveness ratio (ICER) falls below this threshold, the intervention is generally considered cost-effective. Conversely, if the ratio is more than the threshold then the  intervention may be  less cost-effective or not cost-effective.

\begin{table}[htbp]
\centering
\label{tab:data}
\begin{tabular}{|c|c|c|c|}
\hline
\textbf{Intervention} & \textbf{Cost} & \textbf{Averted Cases} \\ \hline
D                     & 1.1928   & 1 \\ \hline
TPT                   & 1.0717   & 1.1379  \\\hline
TPT and D             & 1.0131   & 1.2173  \\\hline
MN                    & 1.0018   & 1.2430 \\\hline
MN and D              & 1.0000  & 1.2786  \\\hline
TPT and MN            &1.0113   & 1.2999  \\\hline
TPT, MN, and D        & 1.0356   & 1.3203               \\\hline          
\end{tabular}
\caption{Scaling Cost column by dividing with 36928638.28 and  scaling Averted cases column by dividing with  77056.52}
\label{tab:scaled table1}
\end{table}

 A comparison between the CET line and frontier line is another way to visualise the cost  effectiveness of the interventions. If CET line is steeper than the cost-effectiveness frontier line  of an intervention then an intervention would be considered cost-effective and if intervention with a CEA frontier steeper than the CET would not be considered cost-effective. A general recommendation by policy makers is that an intervention with the highest ICER value that is less or equal to the CET would be recommended. Interventions that will cost more but will be more effective than the current intervention  is looked upon by the policy makers. \\
 
 A hypothetical study in similiar line to  \cite{Thresholdreporting} is performed to understand the cost effectiveness of different interventions. Scaling is done in order to make a visualisation simpler for the study. It is done by selecting the smallest value from each of the columns i.e. Cost and Averted cases and dividing all the entries of each column by their smallest value. Scaling down a Cost column by dividing with 36928638.28 and scaling down Averted cases column by dividing with 77056.52 in table \ref{tab:sorted averted cases}, we obtain table \ref{tab:scaled table2} which also contains the ICER values of each of the intervention. Since, both the cost and averted cases for each of the intervention are positive, all the data points will fall on the   North East quadrant of  CEA plane. Figure \ref{ThresholdPlot} demonstrates the CEA plane with the different interventions plotted using the values from table \ref{tab:scaled table1}.

\begin{table}[htbp]
\centering

\label{tab:data3}
\begin{tabular}{|c|c|c|c|c|}
\hline
\textbf{Intervention} & \textbf{Cost} & \textbf{Averted Cases} & \textbf{ICER} \\ \hline
D                     & 1.1928   &  1         &   1.192                           \\\hline
TPT                   & 1.0717   & 1.1379     &  -0.87                    \\\hline
TPT and D             & 1.0131   & 1.2173     &  -0.73              \\\hline
MN                    & 1.0018   & 1.2430     &  -0.43         \\\hline
MN and D              & 1.0000  & 1.2786      &  -0.05         \\\hline
TPT and MN            &1.0113   & 1.2999      &  0.53            \\\hline
TPT, MN, and D        & 1.0356   & 1.3203     &  1.190       \\\hline
                         
\end{tabular}
\caption{Intervention with their ICER Values}
\label{tab:scaled table2}
\end{table}

\begin{figure}[ht] 
    \centering
    \includegraphics[width=15cm, height=10cm]{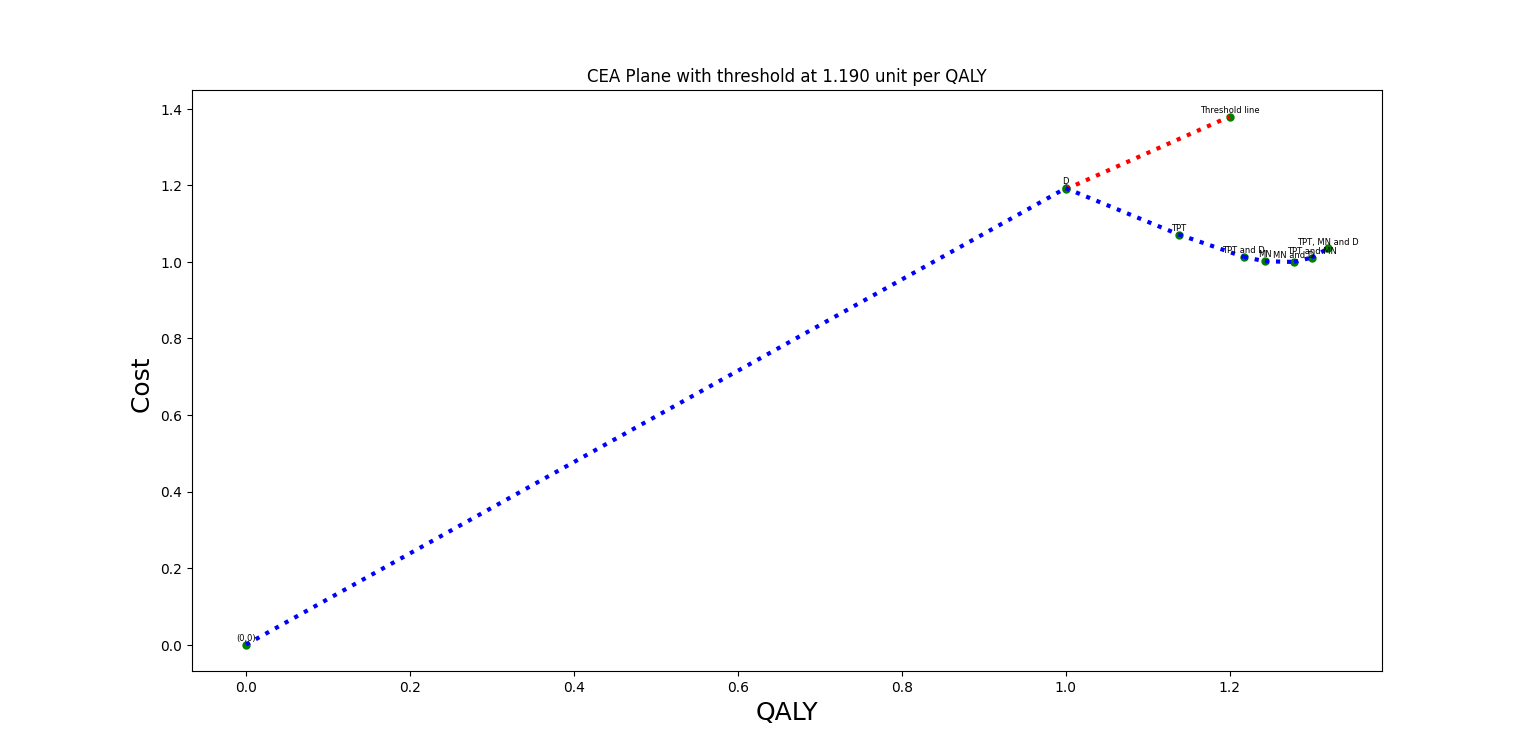}  
    \captionof{figure}{Cost Effectiveness Analysis Plane. }
    \label{ThresholdPlot}
\end{figure}
Hypothetically, we fix the threshold value at 1.190 per QALY. In figure \ref{ThresholdPlot}, frontier lines from origin to D is steeper than CET, indicating  D is not cost effective which is marked by an ICER value of 1.192 that turn out to be greater than CET. This means that D requires more cost than actually one is  willing to pay. Therefore, intervention D is not favourable choice in terms of cost effectiveness with a threshold at 1.190. \\

\begin{figure}[ht] 
    \centering
    \includegraphics[width=15cm, height=10cm]{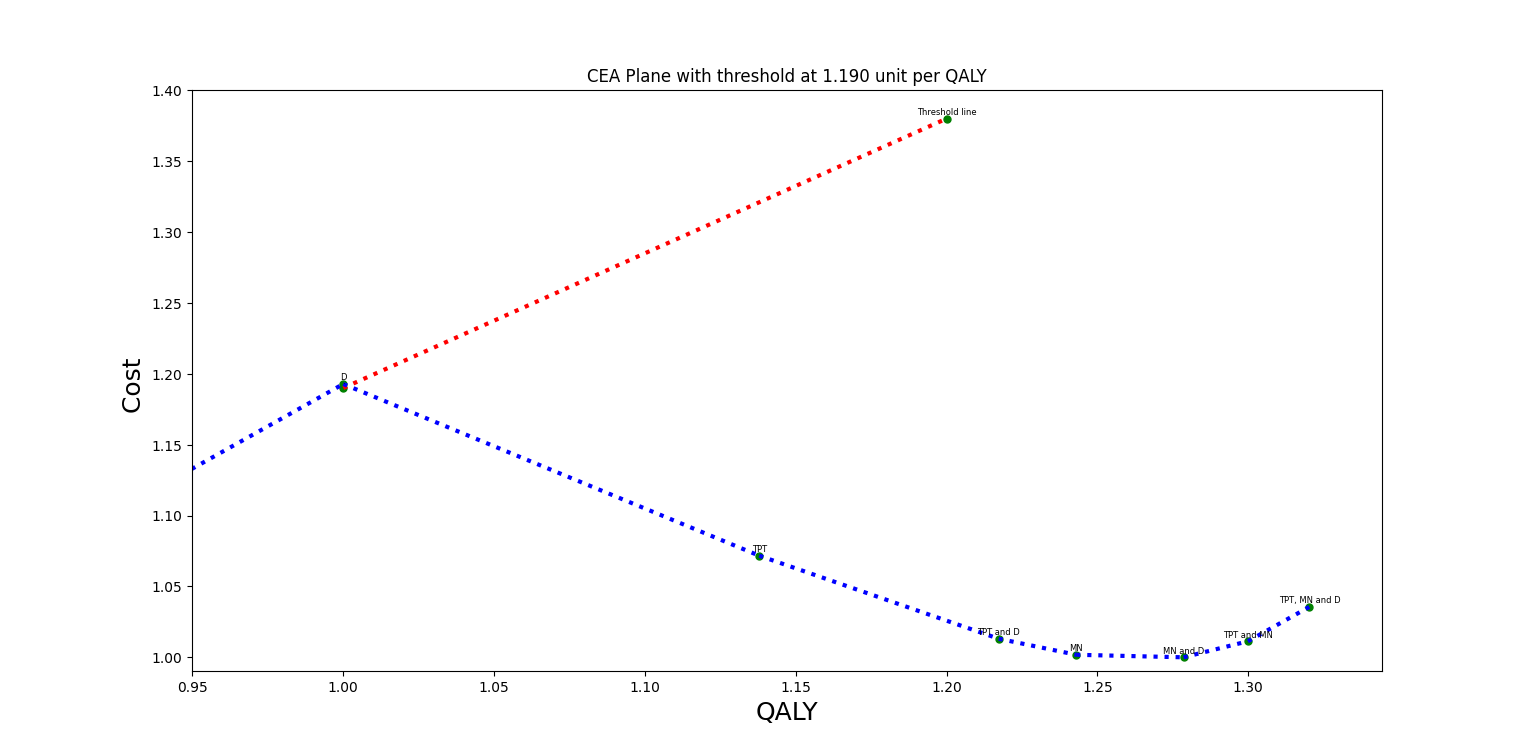}  
    \captionof{figure}{Zoomed Cost Effectiveness Analysis Plane. } 
    \label{ZoomedThresholdPlot}
\end{figure}

  As presented in figure \ref{ZoomedThresholdPlot}, moving from  D to TPT,  TPT to TPT and D, TPT and D to MN and D follow a similar trend where each frontier is shallower than the threshold indicating cost effectiveness. As we move from MN and D to TPT and MN, TPT and MN to TPT,MN and D the frontiers become steeper but still remains shallower than CET. This means that they are costlier than previous interventions but  more cost effective.Therefore, except intervention  D, all the other interventions are cost effective.  Among these interventions, the highest ICER value that is less or equal to the CET (1.190) is that of a multiple intervention TPT, MN and D. Therefore, it is the most cost effective as ICER (TPT, MN and D) is same as threshold.Thus, with this threshold value, one would recommend multiple intervention TPT,MN and D .If we choose threshold to be lesser than 1.190 then we would get frontier of TPT, MN and  D to be steeper than CET and would reject it and look for another intervention X whose frontier would be shallower than CET and ICER standing at or below CET.
 
\section{Discussions and Conclusions}

  In this study, we investigated the efficacy of TPT, MN, and D treatments in treating TB, both when used independently and in combination. The investigation employs two primary methodologies. We employed a thorough approach to these therapies as control measures, which included an extensive review of the optimal control problem. This involves an in-depth understanding of the dynamics of each compartment over the years. Next, we performed a thorough analysis to determine the effectiveness of various combinations of therapies in reducing the occurrence of tuberculosis infection. The data demonstrated that the greatest decrease in the average polluted area was observed when all three interventions were applied simultaneously. This indicates a promising approach to controlling tuberculosis. \\

  Furthermore, in order to assess the effectiveness of each intervention, a cost-effectiveness analysis was conducted. The Average Cost-Effectiveness Ratio (ACER) and Additional Incremental Ratio (AIR) values were calculated for all possible intervention combinations. In order to determine the optimal intervention, a hypothetical cost-effectiveness analysis was conducted. This analysis utilized the four quadrants of the CEA (Cost-Effectiveness Analysis) plane, along with the threshold and ICER (Incremental Cost-Effectiveness Ratio). The implementation of this rigorous methodology guarantees the dependability and replicability of our findings.  \\

  Based on the findings of both the optimal control and comparative effectiveness studies, it can be concluded that the administration of all drugs, whether individually or in combination, results in a substantial decrease in the population of infected compartments and a significant increase in the population of uninfected compartments. Similarly, when calculating the cost-effectiveness of an intervention using the ICER values and the four corners of the CEA plane, the intervention with the lowest total cost is totally determined by the Cost-effectiveness Threshold value. 
 \\

Some of the studies on tuberculosis where the effectiveness of the control strategies are studied can be found in the literature \cite{olaniyi2023fractional, ullah2020optimal, yusuf2023effective}.
The optimal combination of prevention and treatment for co-infection of HIV and TB is documented in the literature \cite{awoke2018optimal}. Despite the fact that there are studies that have been conducted to strategize the ideal control for tuberculosis using various interventions, the combination that was taken into consideration in our study may be one of a kind or extremely uncommon for country-specific situations.
The selection of weight constants linked to control variables can have a substantial impact on the behavior of the control. The primary goal of this study is to investigate the role and efficacy of combined control intervention strategy. Therefore, we will not delve into the examination of the effects of different weight constants. \\

In conclusion, it is our contention that the most optimal approach involves the simultaneous evaluation of all three interventions, as this has been demonstrated to be the most efficacious in reducing infection. It is imperative to acknowledge that the implementation of this integrated methodology may result in potential adverse effects on various bodily organs. Even though there are risks, this plan makes sense when you compare it to problems like kidney damage, heart problems, a weakened immune system, and vision issues that can happen because of diabetes and poor nutrition. In both scenarios, the foremost priority revolves around the patient's survival and overall welfare.

\section*{Data Availability Statement (DAS)}

We do not analyse or generate any datasets, because our work proceeds within a theoretical and mathematical approach.

\section*{Declarations}
The authors declare no Conflict of Interest for this research work.

\section*{Ethics Statement} This research did not require any ethical approval.

\section*{Acknowledgments}
The authors dedicate this paper to the founder chancellor of SSSIHL,  Bhagawan Sri Sathya Sai Baba . The first author dedicate this paper to his loving mother, Manuka Chhetri and brother, Martyr Bikram Chhetri. 
The corresponding author also dedicates this paper to his loving elder brother D. A. C. Prakash who still lives in his heart. The author extends special thanks to Dr. Bishal Chhetri and Mr. Bhanu Praksh.

\printbibliography
\nocite{*}
\end{document}